\documentclass{amsart}
\usepackage[utf8]{inputenc}

\usepackage{amssymb,amsmath,latexsym,amsthm,hyperref}
\numberwithin{equation}{section}

\usepackage{tikz}

\newcommand{\gb}{\mathfrak{b}}

\renewcommand{\gg}{\mathfrak{g}}

\newcommand{\gl}{\mathfrak{l}}

\newcommand{\gn}{\mathfrak{n}}

\newcommand{\gp}{\mathfrak{p}}

\newcommand{\gt}{\mathfrak{t}}
\newcommand{\gu}{\mathfrak{u}}

\newcommand{\B}{\mathbb{B}}
\newcommand{\C}{\mathbb{C}}

\newcommand{\N}{\mathbb{N}}

\newcommand{\R}{\mathbb{R}}

\newcommand{\Z}{\mathbb{Z}}

\newcommand{\cA}{\mathcal{A}}
\newcommand{\cB}{\mathcal{B}}

\newcommand{\cH}{\mathcal{H}}

\newcommand{\cK}{\mathcal{K}}

\newcommand{\cO}{\mathcal{O}}

\newcommand{\cR}{\mathcal{R}}

\newcommand{\cT}{\mathcal{T}}

\theoremstyle{plain}

\newtheorem{theorem}{Theorem}
\newtheorem{lemma}[theorem]{Lemma}
\newtheorem{corollary}[theorem]{Corollary}
\newtheorem{proposition}[theorem]{Proposition}
\newtheorem{definition}[theorem]{Definition}
\numberwithin{theorem}{section}

\theoremstyle{definition}

\newcommand{\GL}{\mathrm{GL}}

\newcommand{\SO}{\mathrm{SO}}

\newcommand{\SU}{\mathrm{SU}}
\newcommand{\U}{\mathrm{U}}

\newcommand{\supp}{\mathop{\operatorname{supp}}}

\DeclareMathOperator{\Ind}{Ind}

\newcommand{\Hom}{\mathrm{Hom}}

\DeclareMathOperator{\Tr}{Tr}

\DeclareMathOperator{\Span}{span}

\newcommand{\gtss}{ {\gt_{\mathrm{ss}}} }
\newcommand{\ggss}{ {\gg_{\mathrm{ss}}} }
\newcommand{\gubar}{ {\overline{\gu} } }
\newcommand{\nbar}{ {\overline{n}} }
\newcommand{\bbar}{ {\overline{b}} }
\newcommand{\gnbar}{ {\overline{\gn}} }
\newcommand{\gbbar}{ {\overline{\gb}} }
\newcommand{\gpbar}{ {\overline{\gp}} }
\newcommand{\Ubar}{ {\overline{U}} }
\newcommand{\Bbar}{ {\overline{B}} }
\newcommand{\Nbar}{ {\overline{N}} }
\newcommand{\Pbar}{ {\overline{P}} }
\newcommand{\hplam}{\cH^{\lambda}}
\newcommand{\alam}{\cA^\lambda}
\newcommand{\op}{\mathrm{op}}
\newtheorem{observation}[theorem]{Observation}
\newcommand{\End}{\mathrm{End}}
\newcommand{\siot}{{\sigma^\lambda\otimes\overline{\sigma^\lambda}}}
\newcommand{\hpiot}{{\cH_\lambda\otimes\overline{\cH_\lambda}}}
\newcommand{\hniot}{{\cH_{\lambda_n}\otimes\overline{\cH_{\lambda_n}}}}
\newcommand{\viot}{{v_\lambda\otimes\overline{v_\lambda}}}
\newcommand{\vniot}{{v_{\lambda_n}\otimes\overline{v_{\lambda_n}}}}
\newcommand{\sniot}{{\sigma^{\lambda_n}\otimes \overline{\sigma^{\lambda_n}}}}
\newcommand{\siox}[1]{{\sigma^\lambda(#1)v_\lambda \otimes \overline{\sigma^\lambda(#1)v_\lambda}}}
\newcommand{\iot}[2]{{#1 \otimes \overline{#2}}}

\allowdisplaybreaks

\begin{document}
\title[Toeplitz Quantization on Flag Manifolds]{A Representation-Theoretic Approach to Toeplitz Quantization on Flag Manifolds}

\author{Matthew Dawson}
\address{SECIHTI---CIMAT Unidad M\'erida,
              Parque Científico y Tecnológico de Yucatán,
              KM 5.5 Carretera Sierra Papacal --- Chuburná Puerto,
              Sierra Papacal; Mérida, YUC 97302, México}

\author{Yessica Hernández-Eliseo}
\address{SECIHTI---CIMAT Unidad M\'erida,
              Parque Científico y Tecnológico de Yucatán,
              KM 5.5 Carretera Sierra Papacal --- Chuburná Puerto,
              Sierra Papacal; Mérida, YUC 97302, México}

\begin{abstract}
In this paper, we study Toeplitz operators on generalized flag manifolds of compact Lie groups using a representation-theoretic point of view.  We prove several basic properties of these Toeplitz operators, including an abstract formula for their matrix coefficients in terms of the decomposition of certain tensor product representations.  We also show how to identify large commuting families of Toeplitz operators based on invariance of their symbols under certain subgroups.  Finally, we realize the Berezin transform as a convolution with certain functions that form an approximate identity on the generalized flag manifold, which allows us to prove a Szegő Limit Theorem using certain results due to Hirschman, Liang, and Wilson.        
\end{abstract}
\maketitle

\section{Introduction}

Over the last century, research in representation theory, quantum mechanics, and harmonic analysis has developed into an intricate dance between analysis, algebra, and geometry.  A striking example of this interplay can be found in the so-called \textit{Berezin-Toeplitz Quantization}, which was first developed in the famous paper \cite{Berezin}, which built on the preexisting theory of Bergman spaces and Toeplitz operators and spurred further research on Toeplitz operators defined on Bergman spaces for bounded symmetric domains. 

In general, Toeplitz operators do not commute, as is desired from a quantization scheme.  
Nevertheless, it was later discovered by R.\ Quiroga and N.\ Vasilevski (see, for instance, \cite{GQV,QV1,QV2}) that there are several large commuting families of Toeplitz operators on the weighted Bergman spaces of the unit disk, and, more generally, the unit ball.  In particular, the unit ball $\B^n$ in $\C^n$ carries a transitive action of the Hermitian Lie group $\SU(n,1)$.  The weighted Bergman spaces $\cA^\lambda = \cO(\B^n) \cap L^2(\B^n, \mu_\lambda)$, for $\lambda>-1$ consist of all holomorphic functions that are square integrable with respect to the measure $d\mu_\lambda(z) = C_\lambda (1-|z|^2)^\lambda dz$ up to a normalizing constant.  The Toeplitz operator with symbol $f\in L^\infty(\B^n)$ is defined to be the operator $T^\lambda_f = B^\lambda M^\lambda_f\in\cB(\cH_\lambda)$, where $B:L^2(\B^n, \mu_\lambda) \rightarrow\cA^\lambda)$ is the orthogonal projection and $M_f:L^2(\B^n, \mu_\lambda) \rightarrow L^2(\B^n, \mu_\lambda)$ is the operator given by multiplication by $f$.  It was shown in \cite{QV1, QV2} that the Toeplitz operators with symbols invariant under a fixed maximal abelian subgroup of $\SU(n,1)$ all commute with each other. These results were proved using symplectic-geometry techniques and several long computations. 

In fact, the Bergman spaces on the unit ball are a special case of the Bergman spaces used by Harish-Chandra to construct the famous Holomorphic Discrete Series representations on the Hermitian Symmetric spaces (see \cite{K} for a standard accounting of these results). It is thus natural to expect that representation-theoretic methods might be a powerful tool for studying Toeplitz operators, and indeed, the results of Quiroga and Vasilevski were extended to general bounded symmetric domains in \cite{DOQ} using purely representation-theoretic techniques. 

At the same time, the construction of Harish-Chandra for the holomorphic discrete series representations of noncompact Hermitian Lie groups also carries through for the compact-type Hermitian symmetric spaces.  This allows Berezin-Toeplitz quantization to be constructed for these contexts, as well.  In fact, Toeplitz operators and the Berezin transform are well-studied on compact-type Hermitian symmetric spaces, and a great deal of explicit and powerful results are known (see, e.g. \cite{EU, EU2}; an earlier example is \cite{Zhang}, in which the spectrum of the Berezin transform for compact Hermitian symmetric spaces is explicitly calculated).  One can also study Berezin-Toeplitz quantization on general compact Kähler manifolds (see, e.g., \cite{Schlichenmaier}), though such general contexts do not lend themselves as easily to representation-theoretic arguments.

For compact groups, the Harish-Chandra construction can be carried out for any generalized flag manifold; in fact, this corresponds precisely to the construction of irreducible representations of compact Lie groups as spaces of sections of certain holomorphic line bundles by the famous Borel-Weil theorem.  The authors are unaware of any previous work that systematically considers Toeplitz operators and Berezin-Toeplitz quantization on flag manifolds using purely representation-theoretic (as opposed to geometric) techniques: this is the goal of the current article.

The paper \cite{HLW} is particularly relevant for the current work.  Its authors considered Toeplitz operators defined on certain finite-dimensional invariant subspaces of $L^2(G/K)$, where $G$ is a compact group and $K$ is a compact subgroup.  These invariant subspaces are assumed to be sums of isotypic subspaces of $L^2(G/K)$; put another way, for each representation $\pi$ of $G$ that occurs in $L^2(G/K)$, each invariant subspace of $L^2(G/K)$ considered in the paper either contains all copies of $\pi$ that occur in $L^2(G/K)$ or contains none. Under certain conditions, certain sequences (or, more generally, nets) of Toeplitz operators assigned to such invariant subspaces of $L^2(G/K)$ are shown to satisfy a Szegő limit theorem.  As written, the sequences of Toeplitz operators we consider in this paper technically do not fit into the scheme of \cite{HLW}.  Nevertheless, we will show that very minimal changes to the central argument of that paper allow one to prove a Szegő limit theorem for Toeplitz operators on generalized flag manifolds.

The scope of the paper is as follows.  In Section~\ref{sec:notation}, we introduce the notation and basic results we need from from representation theory and harmonic analysis on compact Lie groups.  In Section~\ref{sec:Toeplitz_definition}, we introduce Toeplitz operators on Bergman spaces on flag manifolds and prove some basic properties.  In particular, for each irreducible representation $(\sigma^\lambda, \cH_\lambda)$ of a compact Lie group $G$ with highest weight $\lambda$, we will construct Teoplitz operators on $\cH_\lambda$ for each function $f\in L^1(G/T)$, where $T$ is a maximal torus for $G$.  

In Section~\ref{sec:Toeplitz_commuting}, we show that the same criterion from \cite{DOQ} can be used to construct commuting families of Toeplitz operators.  In Section~\ref{sec:matrix_coefficients}, we determine the kernel of the Toeplitz quantization operator and prove an abstract formula for the matrix coefficients of Toeplitz operators. In Section~\ref{sec:Berezin_transform}, we introduce the Berezin transform and prove several basic properties. In particular, we will see that the Berezin transform is given by convolution with the absolute value squared of a conical function, which we will refer to as the Berezin kernel. 

The heart of the paper is Theorem~\ref{thm:approx_identity} in Section~\ref{sec:approximate_identity}, in which it is shown that these Berezin kernels form an approximate identity on $G/T$ (or more generally on generalized flag manifolds $G/L$, where $P\subseteq G^\C$ is a parabolic subgroup and $L=G\cap P$) when we choose a sequence $\{\lambda_n\}_{n\in\N}$ of highest weights that grow to infinity subject to a certain criterion: in particular, if we expand the highest weights as linear combinations of fundamental weights of $G$, we wish for the minimum and maximum of the coefficients over a given subset S of the fundamental weights of $G$ to tend to infinite at least superlogarithmically (in the case of the minimum) and at most polynomially (in the case of the maximum).  The freedom to choose such sequences $\{\lambda_n\}_{n\in\N}$ in such a way that they are not all multiples of a fixed highest weight (which corresponds to taking tensor powers of a fixed line bundle on the flag manifold $G/T$) is presumably a phenomenon new to this work.

Finally, we prove some immediate consequences of Theorem~\ref{thm:approx_identity}, which include the semiclassical limit of the Berezin transforms, a result about asymptotic multiplicities of irreducible representations in a tensor product and, as a consequence of the results in \cite{HLW}, a Szegő Limit Theorem for Toeplitz operators on flag manifolds.

%The inspiration for these ideas can be traced back to the observation by Heisenberg that the Canonical Commutation Relations (CCRs) for the position and momentum operators in non-relativistic quantum mechanics are completely analogous to the Poisson brackets of the corresponding observables in classical mechanics.  Imposing the Canonincal Commutation Relations on the quantum position and momentum operators gives rise to a unitary action of the \textit{Heisenberg group}, whose unitary representations were classified by Stone and von Neumann.  In turn, this classification allows one to derive the Schrödinger equation for the non-relativistic free particle from the CCRs.  

%The analogy between the commutation relations of the quantized position and momentum operators and the Poisson bracks of their classical counterparts might seem mysterious at first and, at any rate, deserve an explanation.  If one accepts that the quantization of the nonrelativistic free particle should carry an irreducible \textit{projective} unitary representation of the Galilean motion group $\Gal(4,\R)$ (which is the full symmetry group of Newtonian mechanics), then the Heisenberg CCRs and the Schrödinger equation for the non-relativistic free particle can be derived, in turn, from the classification of the irreducible projective representations of the Galilean motion group.

%Nevertheless, the geometry of the orbits of 

\section{Notation and Background}
\label{sec:notation}
Let $G$ be a compact, connected Lie group, and let $\gg$ be its Lie algebra.  We fix a maximal torus $T$ with Lie algebra $\gt$. These Lie algebras have respective complexifications $\gg^\C$ and $\gt^\C$, and we denote by $G^\C$ the complexification of $G$.  We denote by $\Delta = \Delta(\gg^\C, \gt^\C)$ the root space of $\gg^\C$ with respect to the Cartan subalgebra $\gt^\C$.  For each $\alpha\in\Delta$, we denote the corresponding root space by $\gg^\C_\alpha\subseteq \gg^\C$.

We recall that $G \subseteq G^\C$ is a maximal compact subgroup, and we denote by $\theta$ the corresponding Cartan involution of $G^\C$. Furthermore, we see that the differentiated Cartan involution on $\gg^\C$ satisfies:
\begin{align*}
   \gg^\C = \gg \oplus i\gg & \rightarrow \gg^\C \\
   X + iY & \mapsto X - iY
\end{align*}
for all $X,Y\in \gg$.  For this reason, we will sometimes use the notation $\theta(X) = \overline{X}$ and $\theta(g) = \overline{g}$ for $X\in\gg^\C$ and $g\in G^\C$.  

Fix a positive subsystem $\Delta^+ \subseteq \Delta$ with corresponding simple roots $\Pi$.  We furthermore define
\[
   \gn = \bigoplus_{\alpha\in\Delta^+} \gg_\alpha^\C, \hspace{1cm} \gnbar = \bigoplus_{\alpha\in\Delta^+} \gg^\C_{-\alpha}.
\]
Note that $\theta(\gn) = \gnbar$ (in fact, $\theta(\gg_\alpha^\C) = \gg_{-\alpha}^\C$ for each $\alpha\in\Delta$) and that $\gg^\C = \gnbar \oplus \gt^\C \oplus \gn$.  We fix the Borel subalgebra $\gb:= \gt^\C \oplus \gn$ and denote the opposite Borel subalgebra by $\gbbar :=\gt^\C \oplus \gnbar$.  Furthermore, we remember that these Lie algebras correspond to the respective closed subgroups $N$, $\overline{N}$, $T^\C$, $B=T^\C N$, and $\overline{B}=T^\C \overline{N}$ of $G^\C$.  Finally, we note that $A=\exp(i\gt)$ is a closed abelian subgroup of $G^\C$ and that $T^\C = TA$. 
The Iwasawa decomposition $G^\C = GA\Nbar$ implies that $G^\C = G\Bbar$ and thus gives rise to the $G$-equivariant diffeomorphism $gT\mapsto g\Bbar$ from $G/T$ to $G^\C/\Bbar$.  

We write $i\gt^* = \Hom_\R (\gt, i\R)$ for the space of purely-imaginary-valued functionals on $\gt$ and note that it can be naturally identified with the space of all complex-linear functionals on $\gt^\C$ that take purely imaginary values on $\gt$.  

We define the semisimple part of $\gg$ by $\ggss=[\gg,\gg]$, and we let $\gtss = \gt \cap \ggss$. Then $\gg = Z(\gg) \oplus \ggss$ and $\gt = Z(\gg)\oplus \gtss$.  The Killing form induces a real inner product on $i\gt^*_{\mathrm{ss}}$, which we denote by $\langle, \rangle$.  By choosing an inner product on $iZ(\gg)$ (which induces an inner product on $iZ(\gg)^*$), we can extend this to an inner product on $i\gt^*$ which we also write $\langle, \rangle$, in such a way that $Z(\gg)\perp \ggss$.  

For each $\alpha\in\Delta$, we define the coroot $\alpha^\vee = \frac{2\alpha}{\langle \alpha, \alpha \rangle}$, thus giving rise to the dual root system $\Delta^\vee$, with corresponding simple roots $\Pi^\vee$.  Denote the simple roots by $\alpha_1, \ldots, \alpha_r$, where $r = \dim \gtss$. Then $\Pi$ and $\Pi^\vee = \{\alpha_1^\vee, \ldots \alpha_r^\vee\}$ are bases for $i\gtss^*$. The fundamental weights $\omega_1, \ldots, \omega_r$ are defined to be the basis for $i\gtss^*$ dual to $\Pi^\vee$ (that is, $\frac{2\langle \omega_i, \alpha_j \rangle}{\langle \alpha_j, \alpha_j\rangle} = \langle \omega_i, \alpha_j^\vee \rangle = \delta_{ij}$ for all $1\leq i\neq j \leq r$). 

Let $\Lambda^+\subseteq i\gt^*$ denote the lattice of dominant, analytically integral weights for $G$ with respect to $T$ and $\Delta^+$.  Then
\[
   \Lambda^+ \subseteq \left\{\left. \sum_{i=1}^r c_i \omega_i \right| c_1,\ldots, c_r \in \Z^{\geq 0} \right\} \oplus iZ(\gg)^*.
\]
\subsection{Parabolic Subalgebras and Flag Manifolds}
\label{sec:parabolic_subgroups}
Each subset $S\subseteq \Pi$ gives rise to a \textit{parabolic subgroup} of $G^\C$ (that is, a subgroup of $G^\C$ that contains a Borel subgroup).
We let $\gt_S \subset \gt$ be such that $i\gt_S^* = \Span\{S\}$.  
Next, let 
\[
   \Gamma = \Delta^+ \cup \{ \alpha\in \Delta \mid \alpha\in \Span{S} \}.
\]
Finally, one defines the parabolic subalgebra 
\[
   \gp = \gt^\C \oplus \bigoplus_{\alpha\in\Gamma} \gg^\C_\alpha. 
\]
One proves that $\Gamma\cap (-\Gamma) \subseteq i\gt^*_S$ is a root system and that  
\[
   \gl^\C := \gt^\C \oplus \bigoplus_{\alpha\in\Gamma\cap(-\Gamma)} \gg^\C_\alpha
\]
is a reductive Lie algebra with Cartan subalgebra $\gt^\C$ and $\Delta(\gl^\C,\gt^\C) = \Gamma\cap(-\Gamma)$.   Furthermore, we choose the positive subsystem $\Delta(\gl^\C,\gt^\C)^+:=\Delta(\gl^\C,\gt^\C) \cap \Delta(\gg^\C,\gt^\C)^+$.  In fact, $\gl^\C$ is a complexification of the real reductive Lie algebra
\[
   \gl = \gt \oplus \bigoplus_{\alpha\in \Delta(\gl^\C,\gt^\C)^+} \R (X_\alpha + \theta(X_\alpha)) \oplus \R i(X_\alpha - \theta(X_\alpha)) ,
\]
where $X_\alpha\in\gg^\C_\alpha \backslash \{0\}$ for all $\alpha\in\Delta$. One shows that $\gl = \gp \cap \gg$.

Let $\gt'_S = \gtss \ominus \gt_S$.  Recalling that $\langle \omega_i , \alpha_j^\vee \rangle = \delta_{ij}$ for all $1\leq i,j \leq r$, we see that
\[
   i(\gt'_S)^* = \Span \{\omega_i \mid \alpha_i \neq S\}.
\]
We have the orthogonal decompositions $\gt_{\mathrm{ss}} = \gt_S'\oplus \gt_S$ and $Z(\gl) = Z(\gg) \oplus \gt'_S$.  It follows that the semisimple part of $\gl$ is
\[
   \gl_{\mathrm{ss}} = \gt_S \oplus \bigoplus_{\alpha\in\Delta(\gl^\C,\gt^\C)^+} \R (X_\alpha + \theta(X_\alpha)) \oplus \R i(X_\alpha - \theta(X_\alpha)) .
\]
We have that $\gl^\C = Z_{\gg^\C}(\gt'_S)$.

We set $\Gamma^+ = \{\alpha\in\Gamma \mid \alpha \notin -\Gamma\}$; then
\[
   \gu = \bigoplus_{\alpha\in\Gamma^+} \gg^\C_\alpha
\]
is a nilpotent Lie algebra and $\gp = \gl^\C \oplus \gu$.  Finally, if we set 
\[
   \gubar = \theta(\gu) = \bigoplus_{\alpha\in\Gamma^+} \gg^\C_{-\alpha},
\]
then 
\[
   \gg^\C = \gubar \oplus \gl^\C \oplus \gu.
\]
We refer to $\gpbar = \gl^\C \oplus \gubar$ as the parabolic subalgebra \textit{opposite} to $\gp$.

We define $U$ and $\Ubar$ to be the analytic subgroups of $G^\C$ with Lie algebras $\gu$ and $\gubar$, respectively.  We let $L = Z_G(\gt'_S)$ and $L^\C = Z_{G^\C}(\gt'_S)$.  Then $P:= LU$ is a parabolic subgroup of $G^\C$ and $\Pbar = L\Ubar$ is its opposite parabolic subgroup, with Lie algebras $\gp$ and $\gpbar$, respectively. 

The compact quotient manifolds $G^\C/P$ and $G^\C/\Pbar$ are called \textbf{partial flag manifolds} and have a natural holomorphic structure. One can show that there are natural $G$-equivariant diffeomorphisms
\begin{align*}
   G/L & \rightarrow G^\C/P  & G/L & \rightarrow G^\C/\Pbar \\
   gL & \mapsto gP      & gL & \mapsto g\Pbar
\end{align*}

\subsection{The Borel-Weil Theorem and Matrix-Coefficient Functions}

Let $\lambda\in i\gt^*$ be a dominant, analytically integral weight. Then $\lambda$ can be integrated to a character $\chi_\lambda: T\rightarrow S^1$ given by $\chi_\lambda(\exp H) = e^{\lambda(H)}$.  This character can be holomorphically extended to a character $\chi_\lambda^\C: T^\C\rightarrow \C^\times$ in the obvious way which can, in turn, be further extended to a holomorphic character $\chi_\lambda^\Bbar: \Bbar\rightarrow \C^\times$ given by $\chi_\lambda^\Bbar(t\nbar) = \chi_\lambda^\C(t)$ for all $t\in T^\C$, $\nbar\in\gnbar$.  When confusion can be avoided, we use the same name $\chi_\lambda$ to refer to the character of $T$ as well its holomorphic extensions to $T^\C$ and to $\Bbar$.

We consider the induced representation $( \Ind_T^G \chi_\lambda, L^2_{\chi_\lambda}(G/T))$, where 
\[
   L^2_{\chi_\lambda}(G/T)  = \{ f\in L^2(G) \mid f(gt) = \chi_\lambda(t^{-1}) f(g) \text{ for almost all } t\in T, g\in G \},
\]
and where $G$ acts, as usual, by left translations.  Note that $\chi_\lambda(t^{-1}) = \chi_{-\lambda}(t) = \overline{\chi_\lambda(t)}$ for all $t\in T$.  Using the Frobenius Reciprocity Theorem, we easily see that
\[
   L^2_{\chi_\lambda}(G/T) \cong \bigoplus_{(\pi, \cH_\pi)\in\widehat{G}} m_\pi \cH_\pi,
\]
where $\widehat{G}$ is the set of all irreducible representations of $G$ (up to equivalence) and where the multiplicity $m_\pi$ is given by $m_\pi = \dim \Hom_T(\chi_\lambda, \pi)$, which is the multiplicity of the weight $\lambda$ in the irreducible representation $(\pi,\cH_\pi)$.

To make this explicit, we define for each unitary representation $(\pi,\cH_\pi)$ of $G$ and each $v,w\in\cH$ the matrix coefficient function:
\begin{align*}
   \pi_{w,v} : G & \rightarrow \C \\
             g & \mapsto \langle w, \pi(g) v\rangle.
\end{align*}
Suppose further that $\pi$ is irreducible and that $v_\lambda \in (\cH_\pi)_\lambda$ is a normalized vector of weight $\lambda$ (not necessarily of highest weight).  Then the map
\begin{align}
   \label{eqn:intertwiningmatrixcoefficient}
   \cH_\pi & \rightarrow L^2_{\chi_\lambda}(G/T) \\
   w & \mapsto \pi_{w,\sqrt{d_{\pi}} v_\lambda} \nonumber
\end{align}
is an intertwining operator that embeds the representation $\pi$ isometrically inside $\Ind_T^G \chi_\lambda$, where $d_{\pi} = \dim \cH_\pi$.  In fact, we see that for each $w\in \cH_\pi$, 
\begin{align*}
   \pi_{w,v_\lambda}(gt) =\,& \langle w , \pi(gt) v_\lambda \rangle \\
                         =\,&\langle w , \chi_\lambda(t) \pi(g) v_\lambda \rangle \\
                         =\,& \overline{\chi_{\lambda}(t)} \langle w , \pi(g) v_\lambda \rangle \\
                         =\,& \chi_{\lambda}(t^{-1}) \pi_{w,v_\lambda}(g) 
\end{align*}
for all $g\in G$ and all $t\in T$.  The Schur orthogonality relations are enough to finish the proof that the intertwining operator in (\ref{eqn:intertwiningmatrixcoefficient}) is an isometry.

Let now $(\sigma^\lambda, \cH_\lambda)$ be the irreducible representation of $G$ with highest weight $\lambda$ 
%(when it is necessary to avoid confusion, we will write $\sigma^\lambda$ in place of $\sigma$)
.  The Borel-Weil theorem identifies the holomorphic sections of the holomorphic line bundle $G\times_T \C_\lambda \cong G^\C \times_\Bbar \C_\lambda$ with the space $\alam = \{ \sigma^\lambda_{w,v_\lambda} \mid w\in \cH_\lambda \}$, where $v_\lambda$ is a normalized highest-weight vector (here $\C_\lambda$ is the copy of $\C$ where the character $\chi_\lambda$ acts).  

In fact, the representation $\sigma^\lambda$ has a unique holomorphic extension $\sigma^\C : G^\C \rightarrow \GL(\cH_\lambda)$.  Furthermore, $(d\sigma)^\C = d(\sigma^\C)$ (that is, the derived representation of $\sigma^\C$ is the complexification of the derived representation of $\sigma^\lambda$).  We thus have that:
\[
   \langle w, d\sigma^\C(X + iY) v_\lambda \rangle = \langle w, d\sigma^\lambda(X) v_\lambda \rangle - i \langle w, d\sigma^\lambda(Y) v_\lambda \rangle
\]
for all $X, Y \in \gg$.  It follows that
\[
   \langle w, d\sigma^\C ( \theta(X+iY) ) v_\lambda \rangle = \langle w, d\sigma^\lambda(\theta X ) v_\lambda \rangle + i \langle w, d\sigma^\lambda(\theta Y ) v_\lambda  \rangle
\]
for all $X, Y \in \gg$, where we use that $\theta(Z) = Z$ for all $Z\in\gg$.  It follows that the map:
\begin{align*}
   \sigma^\C_{w,v_\lambda} : G^\C & \rightarrow \C \\
                       g    & \mapsto \langle w, \sigma^\C( \overline{g} ) v_\lambda \rangle 
\end{align*}
is the unique holomorphic extension of $\sigma^\lambda_{w,v_\lambda}$ to $G^\C$,  where we remind the reader of the notation $\theta(g) = \overline{g}$ for all $g\in G^\C$.  Finally, we see that if $g\in G^\C$, $t\in T$, $a\in A$, and $\nbar \in \Nbar$ (so that $\theta(\nbar) = n\in N$), then:
\begin{align*}
   \sigma^\C_{w,v_\lambda} ( g ta \nbar) = &  \langle w, \sigma^\C(\theta(gta\nbar)) v_\lambda \rangle \\
                                  = & \langle w, \sigma^\C(\overline{g} ta^{-1} n) v_\lambda \rangle \\
                                  = & \langle w, \chi_\lambda^\C(ta^{-1}) \sigma^\C(\overline{g}) v_\lambda \rangle \\
                                  = & \overline{\chi^\C_\lambda(ta^{-1})} \sigma^\C_{w,v_\lambda}(g) \\
                                  = & \chi_\lambda^\C((ta)^{-1}) \sigma^\C_{w,v_\lambda}(g)\\
                               = & \chi_\lambda^\Bbar((ta\nbar)^{-1}) \sigma^\C_{w,v_\lambda}(g), 
\end{align*}
where we use that $\theta(a) = a^{-1}$ and that $\chi_\lambda^\C(a) \in \R$, as well as the fact that $A$ and  $T$commute with each other and normalize $N$.
Thus $\sigma^\C_{w,v_\lambda}$ corresponds to a holomorphic section of the holomorphic line bundle $G^\C \times_\Bbar \C_\lambda $.  In fact, one way of formulating the Borel-Weil theorem is that the space:
\[
   \{f : G^\C \rightarrow \C \mid f\text{ holomorphic}, f(g\bbar) = \chi_\lambda^\Bbar(\bbar)^{-1} f(g) \text{ for all } g\in G^\C, \bbar\in \Bbar \} 
\]
is equal to $\cA^\lambda = \{\sigma^\lambda_{w,v_\lambda} \mid w\in \cH_\lambda\}$ (see \cite[p. 143]{K2}).

From this point on, we will abuse notation and use $\sigma$ to denote both $\sigma$ and $\sigma^\C$.

Finally, we observe that if $G$ and $K\leq G$ are compact groups, then we will make no distinction in this paper between functions on $G/K$ and right-$K$-invariant functions on $G$.  We will further make no distinction between open subsets of $G/K$ and right-$K$-invariant subsets of $G$.
\subsection{Conical Functions}
As before, let $(\sigma^\lambda, \cH_\lambda)$ be an irreducible unitary representation of $G$ with highest weight $\lambda$ and normalized highest-weight vector $v_\lambda\in\cH_\lambda$.  The matrix-coefficient function $\sigma^\lambda_{v_\lambda,v_\lambda}$ satisfies the special property that:
\[
   \sigma^\lambda_{v_\lambda, v_\lambda}(n_1 g n_2) = \langle \sigma^\lambda(n_1)^{-1} v_\lambda, \sigma^\lambda(\overline{g}) \sigma^\lambda(\overline{n_2})) v_\lambda \rangle = \langle v_\lambda, \sigma^\lambda(\overline{g}) v_\lambda \rangle = \sigma^\lambda_{v_\lambda, v_\lambda}(g)  
\]
for all $g\in G^\C$ and all $n_1\in N$, $n_2 \in \Nbar$ (the fact that $v_\lambda$ is a highest-weight vector implies that $\sigma^\lambda(N)v_\lambda = v_\lambda$).  Such functions on $G^\C$ are called \textbf{conical functions}.

If we denote by $L^2(G)_{\mathrm{fin}}$ the space of $G$-finite vectors in $G$, then one can naturally embed $L^2(G)_{\mathrm{fin}}$ inside the space $\cO(G^\C)$ of holomorphic functions on $G^\C$ (one obtains the span of all matrix-coefficient functions of finite-dimensional holomorphic representations of $G^\C$). Then $G^\C \times G^\C$ acts on this space by $(g_1,g_2) \cdot f(x):= f(g_1^{-1} x g_2)$ for all $g_1,g_2,x\in G^\C$.  Furthermore the space of $N\times\Nbar$-invariant functions in $L^2(G)_{\mathrm{fin}}$ is given by the algebraic direct sum 
\[
   L^2(G)_{\mathrm{fin}}^{N\times\Nbar} = \bigoplus^{\mathrm{alg}}_{\lambda\in\Lambda^+} \C \sigma^\lambda_{v_\lambda, v_\lambda}.
\]

Denote by $\Delta_i := \sigma^{\omega_i}_{v_{\omega_i}, v_{\omega_i}}$ the conical function corresponding to the irreducible representation $\sigma^{\omega_i}$ with highest weight given by the fundamental weight $\omega_i$, where $1\leq i\leq r$.  This choice of notation is motivated by the fact that, for the standard choices of maximal torus and simple roots for the groups $\U(n)$ and $\SU(n)$, one has that $\Delta_i(g)$ is the $(n-i)$th principal minor (starting from the bottom right corner) of $g$ for all $g\in \U(n)$.

In what follows, we will denote the conical matrix-coefficient function $\sigma^\lambda_{v_\lambda,v_\lambda}$ by $\Delta^\lambda$ for each $\lambda\in\Lambda^+$.  This is motivated by the following lemma:

\begin{lemma}
   \label{lem:conical_tensor_prod}
   If $\lambda = \lambda_1 \omega_1 + \cdots +\lambda_r \omega_r + \lambda_Z\in \Lambda^+$, where $\lambda_i \in\Z^{\geq 0}$ for $1\leq i \leq r$ and $\lambda_{Z}\in iZ(\gg)^*$, then the corresponding conical function is given by:
   \[
      \Delta^\lambda(g) = \Delta_1(g)^{\lambda_1} \cdots \Delta_r(g)^{\lambda_r} \chi_{\lambda_Z}(g).
   \]
   for all $g\in G^\C$, where $\chi_{\lambda_Z}$ is the central character of $G$ corresponding to the weight $\lambda_Z\in iZ(\gg)^*$.
\end{lemma}
\begin{proof}
   This is a consequence of the fact that if $\lambda, \mu\in \Lambda^+$, then the so-called \textbf{Cartan product representation} $\sigma^{\lambda+\mu}$ has multiplicity one in the tensor product $\sigma^\lambda \otimes \sigma^\mu$, and furthermore a normalized highest-weight vector for the embedded copy of $\sigma^{\lambda+\mu}$ in $\sigma^\lambda \otimes \sigma^\mu$ is given by $v_\lambda \otimes v_\mu$.  It thus follows that

   \[
      \sigma^{\lambda+\mu}_{v_{\lambda+\mu}, v_{\lambda+\mu}} = \sigma^\lambda_{v_\lambda,v_\lambda} \sigma^\mu_{v_\mu,v_\mu}
   \]
   as functions on $G^\C$.
\end{proof}

\begin{lemma}
   \label{lem:stabilizer_of_conical}
   Let $\lambda = \lambda_1 \omega_1 + \cdots + \lambda_r \omega_r + \lambda_Z \in \Lambda^+$ as above. Then
   \[
      \{ g\in G^\C \mid \sigma^\lambda(g) v_\lambda \in \C v_\lambda \} =P 
   \]
   where $P$ is the parabolic subgroup of $G^\C$ corresponding to the subset $S = \{\alpha_i \in \Pi \mid \lambda_i = 0\}\subseteq \Pi$ and $\chi_\lambda \in \widehat{T}$ is the character of $T$ corresponding to the weight $\lambda$.

\end{lemma}
\begin{proof}
   We will use the fact that $v_\lambda$ is a normalized highest-weight vector for $(\sigma^\lambda, \cH_\lambda)$.   Clearly $\sigma^\lambda(B)v_\lambda \in \C v_\lambda$.  Thus, $P=\{g\in G^\C \mid \sigma^\lambda(g) v_\lambda\in \C v_\lambda\}$ must be a parabolic subgroup of $G$ containing $B$.  In order to show that it is equal to the parabolic subgroup $P$ of $G$ for the set of simple roots $\alpha_i$ such that $\lambda_i =0$, we will show,  for all $1\leq i \leq r$, that $d\sigma^\lambda(\gg^\C_{-\alpha_i})v_\lambda \in \C v_\lambda$ if and only if $\lambda_i = 0$. It will follow that $\gg_{-\alpha_i} \subseteq \gp$ if and only if $\lambda_i=0$, which will finish the proof.

   Fix $\alpha_i\in \Pi$, where $1\leq i \leq r$.  Let $X_{\alpha_i}\in \gg_{\alpha_i} \backslash\{0\}$ and let $X_{-\alpha_i}\in \gg_{-\alpha_i}$ such that $[X_{\alpha_i},X_{-\alpha_i}] = H_{\alpha_i^\vee}\in i\gtss$, where $H_{\alpha_i^\vee}\in i\gtss$ is defined by $\gamma(H_{\alpha_i^\vee}) = \frac{2\langle \gamma, \alpha_i\rangle}{\langle \alpha_i, \alpha_i\rangle}$ for all $\gamma\in i\gt^*$. Then
   \begin{align*}
      d\sigma^\lambda(X_{\alpha_i})d\sigma^\lambda(X_{-\alpha_i})v_\lambda & = d\sigma^\lambda([X_{\alpha_i},X_{-\alpha_i}])v_\lambda \\
      & = d\sigma^\lambda(H_{\alpha_i^\vee})v_\lambda  \\
      & = \frac{2\langle \lambda, \alpha_i \rangle}{\langle \alpha_i, \alpha_i \rangle} v_\lambda \\
      & = \lambda_i v_\lambda
   \end{align*}

   It follows that $d\sigma^\lambda(X_{-\alpha_i}) v_\lambda \neq 0$ if $\lambda_i \neq 0$.  Furthermore, $d\sigma^\lambda(X_{-\alpha_i}) v_\lambda$ is a vector of weight $\lambda-\alpha_i \neq \lambda$ and is thus linearly independent from $v_\lambda$.

   Now suppose that $\lambda_i = 0$.  We will show that $d\sigma^\lambda(X_{-\alpha_i})v_\lambda$ is either zero or else a highest-weight vector linearly independent of $v_\lambda$ (which is a contradiction, since $\sigma^\lambda$ is irreducible). 

   Each $\beta\in\Delta^+$ can be written as $\beta=\beta_1 \alpha_1 + \cdots \beta_r \alpha_r$ with $\beta_1,\ldots,\beta_r \in \{0,1,2\}$.  Then, as before
   \[
      d\sigma^\lambda(X_\beta)d\sigma^\lambda(X_{-\alpha_i})v_\lambda =
        d\sigma^\lambda([X_\beta, X_{-\alpha_i}]) v_\lambda.
   \]  

   But $[X_\beta, X_{-\alpha_i}] \in \gg^\C_{\beta-\alpha_i}$. There are now three possibilities: either $\gg^\C_{\beta-\alpha_i}=\{0\}$, or else $\beta-\alpha_i\in \Delta$, or $\beta-\alpha_i=0$.   If $\beta=\alpha_i$, then the above argument shows that $d\sigma^\lambda([X_{\alpha_i},X_{-\alpha_i}])v_\lambda = 0,$ since $\lambda_i = 0$. If $\beta-\alpha_i$ is a root, then it must be a positive root, and thus $d\sigma^\lambda([X_\beta, X_{-\alpha_i}]) v_\lambda = 0$ because $v_\lambda$ is a highest-weight vector. Finally, if $\gg^\C_{\beta-\alpha_i} = \{0\}$, then it is obvious that $d\sigma^\lambda([X_\beta, X_{-\alpha_i}]) v_\lambda = 0$.

   Thus, we have that $d\sigma^\lambda(X_{-\alpha_i}) v_\lambda$ is either zero or is a highest-weight vector for $\sigma^\lambda$ of weight $\lambda-\alpha_i\neq \lambda$.  The latter is impossible because $\sigma^\lambda$ is irreducible, and so we have that:
   \[
      d\sigma^\lambda(X_{-\alpha_i}) v_\lambda = 0
   \]
   for all $1\leq i \leq n$ such that $\lambda_i = 0$.  

   Since $d\sigma^\lambda(X_{-\alpha_i}) v_\lambda = 0$ if and only if $\lambda_i = 0$ for all $1\leq i \leq n$, and since  
   \[
      \{ g\in G^\C \mid d\sigma^\lambda(g) \in \C v_\lambda \},
   \]
   is a parabolic subgroup whose Lie algebra must contain $\gg_{-\alpha_i}$ if and only if $\lambda_i = 0$, we arrive at the conclusion.
\end{proof}

\begin{corollary}
   \label{cor:absolute_value_equals_one}
   In the context of Lemma~\ref{lem:stabilizer_of_conical}, 
   \[
      \{g\in G \mid |\Delta^\lambda(g)| =1 \} = L,
   \]
   where $L = Z_\gg(\gt'_S) = P \cap G$ for the parabolic subgroup corresponding to $S= \{\alpha_i \in \Pi \mid \lambda_i = 0\}$, as before.
\end{corollary}
\begin{proof}
   We note that $||v_\lambda|| = ||\sigma^\lambda(g) v_\lambda|| = 1$, and thus that 
   \[
      |\Delta^\lambda(g)| = |\langle v_\lambda, \sigma^\lambda(g) v_\lambda\rangle | \leq 1
   \]
   for all $g\in G$, with equality holding exactly when $\sigma^\lambda(g) v_\lambda$ is a multiple of $v_\lambda$.  By Lemma~\ref{lem:stabilizer_of_conical}, this holds precisely when $g\in P\cap G = L$.
\end{proof} 

\subsection{Holomorphic Sections of Line Bundles on Partial Flag Manifolds}

Now fix $S\subseteq \Pi$ let $P$ be the corresponding parabolic subgroup of $G^\C$, and let $\Pbar$ be the opposite parabolic group; we put $L = P\cap G = \Pbar \cap G$ as before. Suppose that $\lambda\in \Lambda^+ \cap iZ(\gl)^*$.  From the discussion in Section~\ref{sec:parabolic_subgroups}, it is clear that this happens precisely when $\lambda\in\Span\{\omega_i \mid \alpha_i\notin S\}$. Then the corresponding character $\chi_\lambda\in \widehat{T}$ extends to a central character of $L$ (recall that $T$ is simultaneously a maximal torus for $G$ and for $L$, and thus both groups share the same set $\Lambda$ of analytically integral weights, and every dominant weight for $G$ es dominant for $L$).  This central character can be extended by complexification to a holomorphic central character $\chi_\lambda^\C$ of $L^\C$.  Using the fact that $P$ admits the semidirect product decomposition $\Pbar=L\ltimes \Ubar$ (known as the \textbf{Langlands Decomposition}), $\chi_\lambda^\C$ can be extended to a central character $\chi_\lambda^P$ of $P$.  From this point on we will denote all three of these characters by $\chi_\lambda$ where confusion can be avoided, as before.

We can then construct the line bundle $G\times_L {\C_\lambda}= G^\C\times_\Pbar{\C_\lambda}$.  As before, we identify the $L^2$-sections of this bundle with the space
\[
   L^2_{\chi_\lambda}(G/L) = \{ f\in L^2(G) \mid f(gl) = \chi_\lambda(l^{-1})f(g) \text{ for almost all $g\in G$, $l\in L$} \},
\]
and we consider the induced representation $(\Ind_L^G(\chi_\lambda), L^2_{\chi_\lambda}(G/L))$.  

It is clear that $L^2_{\chi_\lambda}(G/L) \subseteq L^2_{\chi_\lambda}(G/T)$. Furthermore, the orthogonal projection is given by the intertwining 
operator:
\begin{align*}
   L^2_{\chi_\lambda}(G/L) & \rightarrow L^2_{\chi_\lambda}(G/T) \\
              f & \mapsto f^\#, \\
\end{align*}
where 
\[
   f^\#(g) := \int_L f(gl) dl
\]
for all $g\in G$.  

As before, Frobenius Reciprocity implies that 
\[
   L^2_{\chi_\lambda}(G/L) = \bigoplus_{\pi\in\widehat{G}} \Span \{ \pi_{w,v} \mid w,v\in \cH_\pi \text{ and } \pi(l)v = \chi_\lambda(l)v \text{ for all } l\in L \}.
\]
Without much difficulty, one shows that the holomorphic sections of $G\times_{\chi_\lambda} L$ correspond precisely to the holomorphic sections of $G\times_{\chi_\lambda} T$, which we have already identified with 
\[
   \cA^\lambda = \{ \sigma^\lambda_{v,v_\lambda} \mid v\in \cH_\lambda\}
\]
by the Borel-Weil theorem.

We continue working with the same parabolic subgroup $P$, but we now let $\lambda\in\Lambda^+$ be any dominant integral weight, and let $(\sigma, \cH_\lambda)$ be the irreducible representation of $G$ with highest weight $\lambda$.  Then $\sigma|_L$ acts irreducibly on the space $(\cH_\lambda)^U$ of $U$-fixed vectors in $\cH_\lambda$, and in fact this unitary representation of $L$, which we will denote by $\pi$, has highest-weight $\lambda$ (see \cite[Theorem~5.114]{K}).  

\section{Definition and Basic Properties of Toeplitz Operators}
\label{sec:Toeplitz_definition}
Every finite-dimensional Hilbert space of functions is a reproducing kernel Hilbert space.  In fact, the reproducing kernel of $\cA^\lambda$ can be expressed simply in terms of matrix-coefficient functions.  Note first that
\[
   \langle \sigma^\lambda_{w, v_\lambda},  d_{\lambda} \sigma^\lambda_{\sigma^\lambda(x)v_\lambda, v_\lambda} \rangle = \langle w, \sigma^\lambda(x) v_\lambda \rangle = \sigma^\lambda_{w, v_\lambda}(x)
\]
for all $x\in G$, $w\in \cH_\lambda$, where $d_{\lambda} = \dim \cH_\lambda$. 

We put $K^\lambda_{x} = d_{\lambda} \sigma^\lambda_{\sigma^\lambda(x) v_\lambda, v_\lambda}$ for each $x\in G$, and define the \textbf{Bergman Kernel} $K^\lambda :G\times G\rightarrow \C$ by $K^\lambda(x,y) = \langle K^\lambda_y, K^\lambda_x \rangle.$ We note that
\[K^\lambda(x,y) =  d_{\lambda} \sigma^\lambda_{\sigma^\lambda(y) v_\lambda, v_\lambda} (x) = d_{\lambda} \langle v_\lambda, \sigma^\lambda(y^{-1}x) v_\lambda \rangle = d_{\lambda} \Delta^\lambda(y^{-1}x)
\]
for all $x,y\in G$, where we remember that $\Delta^\lambda = \sigma^\lambda_{v_\lambda, v_\lambda} $.  Note that
\[
   K^\lambda(xt_1, yt_2) = \chi_\lambda(t_1t_2^{-1})K^\lambda(x,y)
\]
for all $x,y\in G$ and $t_1,t_2\in T$, so we can also interpret $K^\lambda$ as a positive-definite kernel for sections of the line bundle $G\times_T \C_\lambda$.

We define the \textbf{Bergman projection} to be the orthogonal projection
\[
   P=P^\lambda : L^2_{\chi_\lambda}(G/T) \rightarrow \alam.
\]
Note that since $K^\lambda$ is the reproducing kernel of $\cA^\lambda$, we have that:
\[
   P^\lambda \varphi (x) = \int_G \varphi(y) K^\lambda(x,y) dy = d_{\lambda} \int_G \varphi(y) \Delta^\lambda (y^{-1}x) dx = d_{\lambda}\varphi *  \Delta^\lambda 
\]
for all $\varphi\in\cA^\lambda$ and $x\in G$.

Now let $f\in L^\infty(G/T)$.  We define the multiplication operator
\[
   M_f = M^\lambda_f \in \cB(L^2_{\chi_\lambda}(G/T))
\]
by $M_f \varphi(x) = f(x) \varphi(x)$ for all $\varphi\in L^2_{\chi_\lambda}(G/T)$ and all $x\in G$.  Here, we are using the canonical embeddings of $L^\infty(G/T)$ and $L^2_{\chi_\lambda}(G/T)$ in $L^2(G)$ (this identifies $L^\infty(G/T)$ with the right-$T$-invariant functions in $L^2(G)$).  It is easy to see that $M_f \varphi \in L^2_{\chi_\lambda}(G/T)$ for all $\varphi\in L^2_{\chi_\lambda}(G/T)$ and that $M_f$ is a bounded linear operator satisfying $||M_f||_\op \leq ||f||_\infty$. 

\begin{definition}
The \textbf{Toeplitz operator with symbol $f\in L^\infty(G/T)$} is the operator $T_f^\lambda = P^\lambda M^\lambda_f\in \End(\cA^\lambda)$. When confusion can be avoided, we will write $T_f$ in place of $T_f^\lambda$.
\end{definition}

\begin{observation}
   Suppose that $f\in L^\infty(G/T)$.  Then we see that for all $\varphi, \psi\in \cA^\lambda$,
   \[
      \langle T_f \varphi, \psi \rangle = \langle M_f \varphi, \psi\rangle =  \int_G f(y) \varphi(y) \overline{\psi(y)} dy.
   \]
   Alternatively, for each $\varphi\in\cA^\lambda$, we have that:
   \[
      T_f \varphi (x) = \langle T_f \varphi, K_x \rangle =  d_{\lambda}\int_G f(y) \varphi (y) \overline{\langle \sigma^\lambda(x)v_\lambda, \sigma^\lambda(y) v_\lambda \rangle} dy.
   \]
   Finally, the unitary intertwining operator $\cH_\lambda \rightarrow \cA^\lambda$, $w \mapsto \sigma^\lambda_{w, \sqrt{d_{\lambda}}v_\lambda}$ allows us to transfer the Toeplitz operator $T_f$ to the operator $\widetilde{T_f}$ defined on $\cH_\lambda$ by:
   \[
      \langle \widetilde{T_f} v, w \rangle = d_{\lambda} \langle T_f \sigma^\lambda_{v,v_\lambda}, \sigma^\lambda_{w, v_\lambda} \rangle 
   \]
   for all $v,w\in \cH_\lambda$. When confusion can be avoided, we will abuse notation and write $T_f$ to refer to both $T_f$ and $\widetilde{T_f}$.
\end{observation}

We remind the reader that if $\cH$ is a finite-dimensional Hilbert space, then there is a standard unitary map $\cH \otimes \overline{\cH} \rightarrow \End(\cH) $, where each simple tensor $v\otimes \overline{w}$ is identified with the rank-one operator on $\cH$ given by $(v\otimes \overline{w}) (u) = \langle u, w \rangle v$ for all $u,v,w\in\cH$.  We further remind the reader of that each unitary representation $(\pi, \cH)$ of $G$ gives rise to the integrated $*$-representation of $L^1(G)$ given by 
\[
   \pi(f) := \int_G f(y) \pi(y) dy.
\]

Using integrated representations, we can write Toeplitz operators in a manner that will be convenient for representation-theoretic arguments:
\begin{lemma}
   \label{lem:toeplitz_as_integrated_rep}
   For all $f\in L^\infty(G/T)$, one has that $T_f = d_{\lambda} (\siot)(f) (v_\lambda \otimes \overline{v_\lambda})$, as an operator on $\cH_\lambda$.
\end{lemma}

\begin{proof}
   Using the identification of $\cA^\lambda$ with $\cH_\lambda$, we have that, for all $v,w\in\cH_\lambda$,  
   \begin{align*}
      \langle T_f v, w \rangle = &\, d_{\lambda} \langle T_f \sigma^\lambda_{v, v_\lambda}, \sigma^\lambda_{w,v_\lambda} \rangle \\
      = & d_{\lambda} \int_G f(y) \langle v, \sigma^\lambda(y) v_\lambda \rangle \overline{\langle w, \sigma^\lambda(y) v_\lambda \rangle} dy \\
      = & d_{\lambda} \int_G f(y) \langle v, \sigma^\lambda(y) v_\lambda \rangle \langle \sigma^\lambda(y) v_\lambda, w \rangle dy \\
      = & d_{\lambda} \int_G f(y) \langle \langle v, \sigma^\lambda(y) v_\lambda \rangle \sigma^\lambda(y) v_\lambda , w \rangle dy \\ 
      = & d_{\lambda} \int_G f(y) \langle (\sigma^\lambda(y) v_\lambda \otimes \overline{\sigma^\lambda(y) v_\lambda}) v, w \rangle dy \\
      = & d_{\lambda} \left\langle \int_G f(y) (\siot)(y)( v_\lambda \otimes \overline{v_\lambda} ) dy\, v, w \right\rangle.
   \end{align*}
   We can therefore conclude that 
   \[
      T_f = d_{\lambda} (\siot)(f) (v_\lambda \otimes \overline{v_\lambda}).
   \]
\end{proof}
\begin{corollary}
   \label{cor:trace_of_Toeplitz}
   For all $f\in L^1(G/T)$, one has that
   \[
      \Tr(T_f) = d_\lambda\int_{G} f(x)\, dx
   \]
\end{corollary}
\begin{proof}
   This follows from the fact that 
   \begin{align*}
      \Tr(T^\lambda_f) & = \Tr(d_\lambda (\siot)(f)(\viot)) \\
                       & = d_\lambda \Tr\left(\int_G f(x)\iot{\sigma^\lambda(x)v_\lambda}{\sigma^\lambda(x)v_\lambda} dx \right)\\
                       & = d_\lambda \int_G f(x) \Tr(\iot{\sigma^\lambda(x)v_\lambda}{\sigma^\lambda(x)v_\lambda}) dx\\
                       & = d_\lambda \int_G f(x) dx,
   \end{align*}
   where we use in the last step that $\Tr(\iot{\sigma^\lambda(x)v_\lambda}{\sigma^\lambda(x)v_\lambda}) = ||\sigma^\lambda(x)v_\lambda||^2= 1$.
\end{proof} 

We continue with some further basic properties of Toeplitz operators on flag manifolds:
\begin{lemma}
   \,\begin{enumerate}
      \item For all $f\in L^\infty(G/T)$, one has that $||T_f||_\op\leq ||f||_\infty$. 
      \item The operator 
      \begin{align*}
         T^\lambda: L^\infty(G/T) & \rightarrow \End(\cH_\lambda) \\
         f    & \mapsto T_f
      \end{align*}
      has a unique continuous extension to an operator $T^\lambda: L^1(G/T) \rightarrow \End(\cH_\lambda)$, where $\End(\cH_\lambda)$ is given the trace norm. 
      \item Let $1\leq p \leq\infty$.  For all $f\in L^1(G/T)$, one has that $||T_f||_p \leq d_{\lambda} ||f||_1$, where $||A||_p$ represents the Schatten $p$-norm of $A$ for all $A\in \End(\cH_\lambda)$. 
\end{enumerate}
\end{lemma}

\begin{proof}
   We note that (1) follows immediately from the fact that
   \[
      ||T_f||_\op = ||PM_f||_\op \leq ||M_f||_\op = ||f||_\infty
   \]
   for all $f\in L^\infty(G/T)$.
   
   To prove (2) and (3), we note that if $f\in C(G/T)$, then we can consider the Banach space $\End(\cH_\lambda)$ equipped with the Schatten $p$-norm, where $1\leq p < \infty$.  We immediately have that
   \begin{align*}
      ||T_f||_p & \leq d_{\lambda} \int_G ||f(x) (\siot)(y)(v_\lambda\otimes\overline{v_\lambda})||_p dy \\ 
                & = d_{\lambda} \int_G |f(x)| \, ||\sigma^\lambda(y) v_\lambda \otimes \overline{\sigma^\lambda(y)v_\lambda} ||_p dy\\ 
                & = d_{\lambda} ||v_\lambda||^2 ||f||_1 \\
                & = d_{\lambda} ||f||_1
   \end{align*}
   It follows that the map $T^\lambda: C(G/T) \rightarrow \End(\cH)$ satisfies $||T_f||_1 \leq d_{\lambda} ||f||_1$ and is thus continuous in the relative topology of $C(G/T)$ inside $L^1(G/T)$.  Because $C(G/T)$ is dense in $L^1(G/T)$, we immediately have (2), as well as (3) for the case $p=1$.  
\end{proof} 

\begin{observation}
   In particular, $||T_f||_\op \leq d_{\lambda} ||f||_1$ for all $f\in L^1(G/T)$. Also, $||T_f||_p \leq d_{\lambda} ||f||_p$ for all $1\leq p\leq \infty$ and all $f\in L^p(G/T)$, since $||f||_p\leq ||f||_1$.
\end{observation}

The following easy lemma is analogous to similar results for Toeplitz operators defined on classical Bergman spaces, and among other things motivates the terminology ``Toeplitz Quantization'': real-valued functions are assigned, via a linear map, to self-adjoint operators, and non-negative functions are assigned to non-negative operators, as one should expect from a quantization procedure. 
\begin{lemma}
   \label{lem:toeplitz_basic_properties}
   \,
   If $f\in L^1(G/T)$, one has that $(T_f)^* = T_{\overline{f}}$ Furthermore, if $f\geq 0$, then $T_f \geq 0$.  
\end{lemma}
\begin{proof}
   Note that if $\varphi,\psi\in \alam$ and $f\in L^\infty(G/T)$, then 
   \begin{align*}
      \langle P M_f \varphi,\psi \rangle & = \langle M_f \varphi, \psi \rangle \\
               & = \langle \varphi, M_{\overline{f}} \psi \rangle \\
               & = \langle \varphi, P M_{\overline{f}}\psi \rangle,
   \end{align*}
   where we are using $M^*_f =M_{\overline{f}}$. The fact that (1) holds for all $f\in L^1(G/T)$ then follows by the density of $L^\infty(G/T)$ in $L^1(G/T)$. 
\end{proof} 

% \begin{theorem}[\cite{FH,T}]
%    Let $(\pi,\cH)$ and $(\rho,\cK)$ be irreducible unitary representations of a compact Lie group $G$.  Furthermore, let $v\in\cH$ be a highest-weight vector for $\pi$ and let $w\in\cK$ be a lowest-weight vector $\rho$.  Then $v\otimes w$ is a cyclic vector for $\pi\otimes\rho$ as a representation of $G$.
% \end{theorem}
%
% \begin{corollary}
%    Let $(\pi,\cH)$ be an irreducible unitary representation of a compact Lie group $G$, and let $v\in\cH$ be a highest-weight vector. Then $v\otimes\overline{v}$ is a cyclic vector for $\pi\otimes\overline{\pi}$.
% \end{corollary}

Our next goal is to show that every operator in $\cB(\cH_\lambda)$ is a Toeplitz operator with some symbol in $L^1(G/T)$, though this symbols is not unique.  This is in stark contrast to the case of Toeplitz operators on weighted Bergman spaces of a complex bounded symmetric domain $D$, in which one has that the map $T^\lambda: L^\infty(D)\rightarrow \cB(\cA^\lambda), f\mapsto T^\lambda_f$ is \textit{injective} but not \textit{surjective} for all $\lambda >-1$ (one still has, in that case, that the Toeplitz quantization map $T^\lambda$ has an image that is dense in the strong operator topology, though not in the operator norm topology).  

We begin with the following lemma:

\begin{lemma}
   \label{lem:toeplitz_averaged_symbol}

   If $L\subseteq G$ is the subgroup given by Corollary~\ref{cor:absolute_value_equals_one} for the highest weight $\lambda\in\Lambda^+$, then:
   \[
      d_{\lambda} (\siot) (f) (v_\lambda \otimes \overline{v_\lambda}) = T_{f^\#_L},
   \]
   where $f^\#_L \in L^1(G/L)$ is given by:
   \[
      f^\#_L(x) = \int_L f(xl) dl.
   \]
   In particular, since one always has that $T\subseteq L$, we see that
   \[
      d_{\lambda} (\siot) (f) (v_\lambda \otimes \overline{v_\lambda}) = T_{f^\#}
   \]
   for all $f\in L^1(G)$, where we denote by $f^\# \in L^1(G/T)$ the function defined by
   \[
      f^\#(x):= \int_T f(xt) dt.
   \]

\end{lemma}
\begin{proof}
   First, we note that the vector $v_\lambda \otimes \overline{v_\lambda}$ is $L$ invariant since $\sigma^\lambda(l)v_\lambda \otimes \overline{\sigma^\lambda(l) v_\lambda} = \chi_\lambda(l)\overline{\chi_\lambda(l)} v_\lambda \otimes \overline{v_\lambda} = v_\lambda \otimes \overline{v_\lambda}$ for all $l\in L$. 

   Thus
   \begin{align*}
      T_{f^\#_L} = &\, d_{\lambda} \int_G \left(\int_L f(yl) dl\right)\, \sigma^\lambda(y) v_\lambda \otimes \overline{\sigma^\lambda(y)v_\lambda} dy\\
      = & \, d_{\lambda} \int_L \int_G f (yl) \sigma^\lambda(yl) v_\lambda  \otimes \overline{\sigma^\lambda(yl) v_\lambda} dy dl\\
      = & \, d_{\lambda}  \int_G f(y) \sigma^\lambda(y) v_\lambda \otimes \overline{\sigma^\lambda(y) v_\lambda} dy \\
      = & \, d_{\lambda} (\siot)(f) (v_\lambda \otimes \overline{v_\lambda}).
   \end{align*}
\end{proof} 
\begin{observation}
   As a consequence of the preceding lemma, we will usually only consider symbols on $G/L$ rather than on $G/T$ in what follows, since one has that 
   \[
      T_f = T_{f^\#_L}.
   \]
   for all $f\in L^1(G/T)$.
\end{observation}

Next, we need the following useful result:
 \begin{theorem}[\cite{FH,T}]
    Let $(\pi,\cH)$ and $(\rho,\cK)$ be irreducible unitary representations of a compact Lie group $G$.  Furthermore, let $v\in\cH$ be a highest-weight vector for $\pi$ and let $w\in\cK$ be a lowest-weight vector $\rho$.  Then $v\otimes w$ is a cyclic vector for $\pi\otimes\rho$ as a representation of $G$.
 \end{theorem}
This tells us, in particular, that $\viot$ is a cyclic vector for the representation $(\siot, \hpiot)$, where, as before, $v_\lambda$ is a highest-weight vector for $(\sigma^\lambda, \cH^\lambda)$.

We are now ready to prove the promised result on the surjectivity of the Toeplitz quantization operator:
\begin{proposition}
   \label{prop:all_operators_are_Toeplitz}
   The linear operator
   \begin{align*}
      T^\lambda: C^\infty(G/L) & \rightarrow \End(\cH_\lambda) \\
            f & \mapsto T^\lambda_f
\end{align*}
   is surjective. 
\end{proposition}
\begin{proof}
   Let $\{\widetilde{\varphi}_n\}_{n\in\N} \subseteq C^\infty(G)$ be an approximate identity for $L^1(G)$ with $\widetilde{\varphi}_n \geq 0$ and $\int_G \widetilde{\varphi}_n(y) dy = 1$ for all $n\in\N$.  Then define $\varphi_n \in C^\infty(G/L)$ by $\varphi_n(x) = \int_L \widetilde{\varphi}_n (xt) dt$ for all $x\in G$.  
   Since $\{\widetilde{\varphi}_n\}_{n\in\N}$ is an approximate identity for $G$, we have that:
   \[
      (\siot)(\widetilde{\varphi}_n) (\viot) \stackrel{n\rightarrow\infty}{\longrightarrow} \viot.
   \]
   But by Lemma~\ref{lem:toeplitz_averaged_symbol}, we have that $T_{\varphi_n} = d_{\lambda} (\siot)(\widetilde{\varphi}_n) (\viot)$ for all $n\in\N$.   
   In particular, $T_{\varphi_n} \rightarrow d_{\lambda} (\viot)$ as $n\rightarrow\infty$. 

   Since $\End(\cH_\lambda)$ is a finite-dimensional vector space, there must exist $\varphi\in C^\infty(G)$ such that $T_\varphi = d_{\lambda} (\siot)$.  In the next section (Lemma~\ref{lem:toeplitz_invariant_symbol}), we will show that $T_{L(x)\varphi} = \sigma^\lambda(x)T_\varphi\sigma^\lambda(x)^{-1}$ for all $\varphi\in C^\infty(G/L)$, $x\in G$.  Therefore, we know that $\{ T_f \mid f\in C^\infty(G/L)\}$ is a $G$-invariant subspace of $\End(\cH_\lambda) = \cH_\lambda \otimes \overline{\cH_\lambda}$.  Since $\viot$ is a cyclic vector for $\siot$, the result follows immediately.
\end{proof} 

\section{Invariant Symbols and Commutative Toeplitz Algebras}
 \label{sec:Toeplitz_commuting}
In this section, we study the question of how to find large families of commuting Toeplitz operators.  As in \cite{DOQ}, we will see that the key is to consider Toeplitz operators with symbols invariant under a subgroup $H$ of $G$ such that the representation $\sigma^\lambda|_H$ is multiplicity free; this essentially works because of Schur's lemma, combined with some averaging arguments.

As before, we let $L$ be the subgroup of $G$ given by \ref{cor:absolute_value_equals_one} for the highest weight $\lambda$. To begin, we note that the Toeplitz quantization map $T^\lambda:L^1(G/L) \rightarrow \End(\cH_\lambda)$ is a $G$-intertwining operator from the quasi-regular representation on $L^1(G/L)$ to the representation $\siot$ on $\End(\cH_\lambda) \cong \hpiot$:

\begin{lemma}
   \label{lem:toeplitz_invariant_symbol}
   If $f\in L^1(G/L)$ and $x\in G$, we denote by $L(x)f : G/L \rightarrow \C$ the function defined by $L(x)f(y) = f(x^{-1} \cdot y)$.  Then $T_{L(x)f} = \sigma^\lambda(x) T_f \sigma^\lambda(x)^{-1}$.
\end{lemma}

\begin{proof}
   If $f\in L^1(G/L)$, then 
   \begin{align*}
      T_{L(x)f} & = d_{\lambda} \int_G f(x^{-1}y) \sigma^\lambda(y)v_\lambda \otimes \overline{\sigma^\lambda(y)v_\lambda} \, dy \\ 
      & = d_{\lambda} \int_G f(y) \sigma^\lambda(xy) v_\lambda \otimes \overline{\sigma^\lambda(xy) v_\lambda}\, dy \\
      & = d_{\lambda} \int_G f(y) \sigma^\lambda(x) (\sigma^\lambda(y) v_\lambda \otimes \overline{\sigma^\lambda(y) v_\lambda}) \sigma^\lambda(x)^{-1} \, dy \\
      & = d_{\lambda} \sigma^\lambda(x) \int_G f(y) (\sigma^\lambda(y) v_\lambda \otimes \overline{\sigma^\lambda(y) v_\lambda})\, dy \, \sigma^\lambda(x)^{-1}  \\
      & = \sigma^\lambda(x) T_f \sigma^\lambda(x)^{-1}.
   \end{align*}
\end{proof}

We now have the following immediate corollary, which shows that the Toeplitz operators with $H$-invariant symbols are $H$-intertwining operators are precisely those with $H$-invariant symbols:
\begin{corollary}
   If $H\leq G$ and $f\in L^1(G/L)$, then $T_f$ is in the $*$-algebra $\End_H(\cH_\lambda)$ of intertwining operators for $\sigma_H$ if $f$ is $H$-invariant---that is, $L(h)f = f$ for all $h\in H$.  We write $L^1(G/L)^H$ (and similarly $C^\infty(G/L)^H$, $L^\infty(G/L)^H$, etc.) for the space of $H$-invariant functions in $L^1(G/L)$.
\end{corollary}

In fact, it turns out that \textit{all} $H$-intertwining operators on $\cH_\lambda$ are given by Toeplitz operators with $H$-invariant symbols:
\begin{proposition}
   Let $H$ be a closed subgroup of $G$.  Then the vector space 
   \[
      \cT^\lambda_H := T^\lambda(C^\infty(G/L)^H) = \{ T^\lambda_f \mid f\in C^\infty(G/L)^H \}
   \]
   of Toeplitz operators with $H$-invariant symbols is equal to the $*$-algebra $\End_H(\cH_\lambda)$ of intertwining operators for $\sigma|_H$.  
\end{proposition}
\begin{proof}
   Let $f\in C^\infty(G/L)$.  We define $f^H\in C^\infty(G/L)^H$ by:
   \[
      f^H(x)= \int_H f(hx) dh \\
   \]
   for all $x\in G$. Then
   \begin{align*}
      T_{f^H} & = d_{\lambda} \int_G \int_H f(hx) dh\, \sigma^\lambda(x)v_\lambda \otimes \overline{\sigma^\lambda(x)v_\lambda} dh\, dx\\
         & = d_{\lambda} \int_H \int_G f(x) \sigma^\lambda(h^{-1}x)v_\lambda \otimes \overline{\sigma^\lambda(h^{-1}x)v_\lambda} dx\, dh \\
         & = d_{\lambda} \int_H \int_G f(x) \sigma^\lambda(hx)v_\lambda \otimes \overline{\sigma^\lambda(hx)v_\lambda} dx\, dh\\
         & = \int_H \sigma^\lambda(h) T_f \sigma^\lambda(h)^{-1} dh.
   \end{align*}

   Now let $S\in \End_H(\cH_\lambda)$ be an intertwining operator for $\sigma|_H$.  By Proposition~\ref{prop:all_operators_are_Toeplitz}, we know that there is $\varphi\in C^\infty(G/L)$ such that $S = T_\varphi$.  Then we have that:
   \begin{align*}
      S & = \int_H \sigma^\lambda(h) S \sigma^\lambda(h)^{-1} dh \\
        & = \int_H \sigma^\lambda(h) T_\varphi \sigma^\lambda(h)^{-1} dh \\
        & = T_{\varphi^H}.
   \end{align*} 
   Since $\varphi^H$ is $H$-invariant, we are done.
\end{proof}

In fact, it does not matter whether one chooses smooth, $L^1$, or $L^\infty$ symbols.
\begin{corollary}
   Let $H$ be a closed subgroup of $G$.  Then 
   \[
      T^\lambda(C^\infty(G/L)^H) = T^\lambda(L^\infty(G/L)^H) = T^\lambda(L^p(G/L)^H) 
   \]
   for all $1\leq p <\infty$.
\end{corollary}

We have now proved the following result, which is essentially identical to Theorem~6.4 in \cite{DOQ}, which was proved in the context of complex bounded symmetric domains. 
\begin{theorem}
   Let $H$ be a closed subgroup of $G$.  Then the Toeplitz operators with $H$-invariant symbols commute if and only if $\sigma^\lambda|_H$ is a multiplicity-free representation.  
\end{theorem}

In particular, if $\sigma^\lambda|_H$ is multiplicity free, then we see that for each symbol $f\in L^1(G/L)^H = L^1(H\backslash G/L)$, the Toeplitz operator acts by a multiple of the identity on each $H$-invariant subspace of $\cH_\lambda$.

We now take stock of the situation: there are several multiplicity-free results known for the restriction of irreducible representations of a compact Lie group $G$ to certain subgroups $H$.  One has, for instance the classical branching laws, which are multiplicity-free for branching from $\U(n)$ to $\U(n-1)$ and from $\SO(n)$ to $\SO(n-1)$.  Additionally, T.\ Kobayashi has proved (see \cite{Kobayashi1,Kobayashi2}) that whenever the highest weight $\lambda\in\Lambda^+$ is such that the subgroup $L$ given by Corollary~\ref{cor:absolute_value_equals_one} is a symmetric subgroup of $G$ (i.e., such that $G/L$ is a symmetric space), then the restriction $\sigma^\lambda|_L$ is multiplicity free.  These representations are called \textit{pan-type} representations by Kobayashi, and he classified them in the cited papers.  

Let us briefly consider the case of $G=\U(n)$ and $H=\U(n-1)$. Let $\lambda\in\Lambda^+(\U(n))$ be a dominant integral weight. Then the Toeplitz operators on $\cH_\lambda$ with $\U(n-1)$ invariant symbols all commute. In fact, we further know that for any symbol $f\in L^1(\U(n)/T)^{\U(n-1)}$, the Toeplitz operator $T^\lambda_f$ acts as a constant multiple of the identity on each $\U(n-1)$-invariant subspace of $\cH_\lambda$.  In fact, in this particular case one can say a lot about the structure of the space $L^2(\U(n)/T)^{\U(n-1)} = L^2(\U(n-1)\backslash \U(n)/T)$, which consists of functions on $\U(n)$ that are right-invariant under $T$ and left-invariant under $\U(n-1)$.  In this case, there is a natural $\U(n)$ equivariant diffeomorphism $\U(n)/\U(n-1) \rightarrow S^{2n-1}$. In fact, it turns out that $\U(n)/\U(n-1)$ is a commutative space (although it is not a symmetric space), so the quasi-regular representation on $L^2(\U(n)/\U(n-1))$ is multiplicity free.

% We define, for each $p,q\in\Z^{\geq 0}$, the space $\cH_{p,q}$ of all harmonic polynomials $p:\C^n\rightarrow\C$ in $z$ and $\overline{z}$ in $\C^n$ that are homogeneous of degree $p$ in $z\in\C^n$ and of degree $q$ in $\overline{z} \in\C^n$, and we let $\U(n)$ act by left translation, giving the representation $\sigma^{p,q}$.  One shows (see \cite{K}) that this is the irreducible representation of $\U(n)$ with highest weight $\lambda=(p,\cdots, -q)$.  Furthermore, it is not difficult to see that the representation on $\cH_{p,q}$ has a $T$-invariant vector exactly when $p=q$ (in fact, the center of $\U(n)$ acts with a character of weight $p-q$).  For each $p\in\Z^{\geq 0}$, let $v_p$ be a unit $\U(n-1)$-invariant vector in $\cH_{p,p}$.  Because $\U(n)/\U(n-1)$ is a commutative space, this vector is unique up to a scalar of modulus one.  One has that:
% \[
%    L^2(\U(n-1)\backslash \U(n)/T) \cong \bigoplus_{p=0}^\infty \cM_p,
% \]
% where $\cM_p = \{ \langle v_{p,p}, \sigma^{p,p}(\cdot) w \rangle L^2(G) \mid w\in (\cH_{p,q})^T\}$. 
%
%

\section{A Tensor Product and the Kernel of the Toeplitz Quantization Map}
\label{sec:matrix_coefficients}
In Section~\ref{sec:Toeplitz_definition}, we showed that the Toeplitz quantization map is surjective.  In this section, we determine the kernel of the Toeplitz quantization map.  In order to do so, we need to take a closer look at the tensor product representation $\siot$. 

The representation $\siot$ can be decomposed as a direct sum of irreducible representations.  We use the notation $\rho \preceq \pi$ to denote that a unitary representation $\rho$ of $G$ is equivalent to a subrepresentation of the unitary representation $\pi$.  We furthermore denote by $m(\rho,\pi):=m_G(\rho,\pi)$ the multiplicity of $\rho$ in $\pi$. A closed formula for these multiplicities in terms of the highest weight $\lambda$ of the representation $\sigma$, which is related to the famous Kostant formula for weight multiplicities, may be found in \cite{Steinberg}.

We can write
\begin{equation}
   \label{eq:tensor_decomp_space}
   \hpiot = \bigoplus_{\rho\preceq\siot} m(\rho,\siot) \cH^\rho,
\end{equation}
where $m(\rho,\siot)\cH^\rho = (\cH^\rho)^{\oplus m(\rho,\siot)}$ is a direct sum of $m(\rho,\siot)$ copies of $\cH^\rho$.  For each vector $w\in\hpiot$, we write:
\begin{equation}
   \label{eq:tensor_decomp_vectors}
   w = \sum_{\rho\preceq\siot} \sum_{k=1}^{m(\rho,\siot)} w^\rho_k,
\end{equation}
where each $w^\rho_k\in \cH^\rho$ represents the projection onto the $k$th copy of $\cH^\rho$ in the decomposition of $\hplam$.

While not strictly necessary for what follows, it is useful to note that one can embed the representation $\hpiot$ canonically as a subrepresentation of $L^2(G/T)$ in the following way: since $\viot$ is a cyclic vector for $\siot$, it follows that the operator
\begin{align*}
   C: \hpiot & \rightarrow L^2(G/T) \\
   \iot vw & \mapsto d_\lambda\langle \iot vw, \siot(\cdot) \viot \rangle_{\hpiot} 
   = d_\lambda \sigma^\lambda_{v,v_\lambda} \overline{\sigma^\lambda_{w,v_\lambda}}%\langle v, \sigma^\lambda(\cdot) v_\lambda\rangle \overline{ \langle w, \sigma^\lambda(\cdot) v_\lambda\rangle} 
\end{align*}
is an isometric $G$-intertwining operator from $\siot$ to the quasi-regular representation on $L^2(G/T)$. 

In fact, this embedding can be refined slightly.  Suppose that
\[
   P = \{g\in G^\C \mid \sigma^\lambda(g) v_\lambda \in \C v_\lambda \}
\]
is the parabolic subgroup constructed by Lemma~\ref{lem:stabilizer_of_conical}.  One quickly sees that, in fact, if $l\in L=P\cap G$, then 
\begin{align*}
   C(\iot vw)(xl) & = \langle v, \sigma^\lambda(xl) v_\lambda \rangle \overline{\langle w, \sigma^\lambda(xl) v_\lambda \rangle }\\
                  & = \chi_\lambda(l)^{-1}\chi_\lambda(l) \langle v, \sigma^\lambda(x) v_\lambda \rangle \overline{\langle w, \sigma^\lambda(x) v_\lambda \rangle} \\
                  & = C(\iot vw)(x)
\end{align*}
for all $x\in G$ and $l\in L$. We see that, in fact, the image of $C$ falls in $L^2(G/L)$:
\[
   C: \hpiot \rightarrow L^2(G/L).
\]
Furthermore $C(\viot) = d_\lambda |\Delta^\lambda|^2 \in L^2(L\backslash G / L)$.  

As a consequence, one has that $\siot$ is a multiplicity-free representation whenever $G/L$ is a commutative space.  This occurs, for instance, if $G/L$ is a symmetric space. The most well-studied case is that of a compact Hermitian symmetric space, in which case $P\leq G^\C$ is a maximal parabolic subgroup. In the case of $G=\U(n)$, this corresponds to representations with highest weight $\lambda = k\omega_i$ for some $k\in\N$ and $1\leq i\leq n$. More generally, representations with highest weight such that $L$ is a symmetric subgroup of $G$ were given the name ``pan-type representations'' by T.\ Kobayashi and were classified in \cite{Kobayashi1} (see also \cite{Kobayashi2}), as we mentioned above.   

In general, $\siot$ is not a multiplicity-free representation, and thus the decomposition of $\siot$ in irreducible representations is not unique, so the decompositions in (\ref{eq:tensor_decomp_space}) and (\ref{eq:tensor_decomp_vectors}) are not canonical; more precisely, the projection $w^\pi$ of a vector $w\in\hpiot$ onto the space $(\hpiot)^\pi$ of $\pi$-isotypic vectors \textit{is} canonical, but the decomposition of $(\hpiot)^\pi$ as a sum of copies of $\cH_\pi$ is not.  Nevertheless, for convenience we will assume for the moment that all tensor product representations of the form $\cH_\lambda$ have been decomposed as a sum of irreducible representations; some results, such as Proposition~\ref{prop:kernel_of_meta_Toeplitz}, do not depend on the choice of decomposition. 

Our next step is to provide a formula for the matrix coefficients of Toeplitz operators in terms of the decompositions (\ref{eq:tensor_decomp_space}) and (\ref{eq:tensor_decomp_vectors}). 

\begin{proposition}
   \label{prop:Toeplitz_matrix_coefficient}
   Using the notation above, if $\pi\preceq\siot$, $v_0\in \cH_\pi^L = \{v\in \cH_\pi | \pi(L)v = v\}$, and $u\in\cH_\pi$ then
   \[
      \langle v,  T_{\pi_{u,v_0}} w \rangle = \frac{d_{\lambda}}{d_{\pi}} \sum_{k=1}^{m(\pi,\siot)} \langle (v\otimes\overline{w})^\pi_k, u \rangle \langle v_0, (\viot)^\pi_k \rangle 
   \]
for all $v,w\in\cH_\lambda$.  
\end{proposition}

\begin{proof}
   Thus, if we let $\pi\preceq\siot$ and we let $v_0\in\cH_\pi^L = \{v\in \cH_\pi | \pi(L)v = v\}$, then:
   \begin{align*}
      \allowdisplaybreaks
      \langle v,  T_{\pi_{u,v_0}} w \rangle & = d_{\lambda} \int_G \overline{\pi_{u,v_0}(x)} \langle v,  (\siox{x}) w \rangle dx \\
             & = d_{\lambda} \int_G \overline{\pi_{u,v_0}(x)}  \langle v, \langle w, \sigma^\lambda(x) v_\lambda \rangle \sigma^\lambda(x) v_\lambda \rangle dx\\
             & = d_{\lambda} \int_G \overline{\pi_{u,v_0}(x)}  \langle v, \sigma^\lambda(x) v_\lambda\rangle \overline{\langle w, \sigma^\lambda(x)v_\lambda \rangle} dx\\
             & = d_{\lambda} \int_G \langle v\otimes{\overline{w}} , \siot (x) (\viot) \rangle \overline{\pi_{u,v_0}(x)} dx \\
             %\ind_{\lambda} t_G \pi_{w,v_0}(x) \langle \siot(x)\left(\sum_{\rho\preceq\siot} \sum_{k=1}^{m(\rho,\siot)} (\viot)^\rho_k\right)v,w\rangle \\
             & = d_{\lambda} \sum_{\rho\preceq \siot} \sum_{k=1}^{m(\rho,\siot)} \int_G\langle (v\otimes\overline{w})^\rho_k, \rho(x) (\viot)^\rho_k \rangle \overline{\pi_{u,v_0}(x)} dx \\
             & = d_{\lambda} \sum_{k=1}^{m(\pi, \siot)} \langle \pi_{(v\otimes\overline{w})^\pi_k, (\viot)^\pi_k}, \pi_{u,v_0} \rangle \\
             & = \frac{d_{\lambda}}{d_{\pi}} \sum_{k=1}^{m(\pi,\siot)} \langle (v\otimes\overline{w})^\pi_k, u \rangle \langle v_0, (\viot)^\pi_k \rangle. 
   \end{align*}

\end{proof}

We notice that $(\viot)^\rho_k \neq 0$ for each $\rho \preceq \siot$ and all $k\in\{1,\ldots,m(\pi,\siot)\}$, since $\viot$ is a cyclic vector for $\siot$.  In fact, we can conclude that 
\begin{equation}
   \label{eqn:linearly_independent_weight_zero_set}
   (\viot)^\pi_1, \ldots, (\viot)^\pi_{m(\pi,\siot)}
\end{equation}
is a linearly-independent subset of $\cH_\pi$.  Furthermore, the fact that $\viot$ is an $L$-invariant vector implies that its projection onto each irreducible factor of $\siot$ is also $L$-invariant.  Thus, the vectors in~(\ref{eqn:linearly_independent_weight_zero_set}) are contained the space $\cH_\pi^L$ of $L$-invariant vectors in $\cH_\pi$.   

We define the vector space 
\[
   \cK^\lambda_\pi = \Span\{(\viot)^\pi_k \mid k=1,\ldots m(\pi,\siot)\} \subseteq \cH_\pi^L
\]
and note that $\dim\cK^\lambda_\pi = m_G(\pi, \siot) \leq \dim \cH_\pi^L$.

If $(\pi,\cH_\pi)$ is a unitary representation of $G$ and $v_0\in \cH_\pi^L \ominus \cK^\lambda_\pi$ (that is, $v\in \cH_\pi^L$ and $v\perp \cK^\lambda_\pi$), then clearly $\pi_{u,v_0}\in \ker T^\lambda$. We now define:
\[
   \cK^\lambda = \Span\{\pi_{u,v_0} \mid \pi \preceq \siot, u\in\cH_\pi, v_0\in\cK^\lambda_\pi\}
\]

\begin{proposition}
   \label{prop:kernel_of_meta_Toeplitz}
   The kernel of $T^\lambda$ is $\ker T^\lambda = (\cK^\lambda)^\perp$.  In particular, 
\[
   T^\lambda: \cK^\lambda \rightarrow \hpiot = \End(\hplam)
\]
is a linear isomorphism.

\end{proposition}
\begin{proof}
   First we note that $(\cK^\lambda)^\perp\subseteq \ker T^\lambda$ as a direct consequence of Proposition~\ref{prop:Toeplitz_matrix_coefficient}.  Finally, we observe that:
   \begin{align*}
      \dim \cK^\lambda & = \sum_{\pi\preceq\siot} \dim\cK^\lambda_\pi \dim \cH_\pi \\
                       & = \sum_{\pi\preceq\siot} m_G(\pi,\siot) d_{\pi}\\
                       & = d_{\siot} = \dim\ (\hpiot).
   \end{align*}
   This, together with the fact that $T^\lambda$ is surjective, finishes the proof.
\end{proof}

\section{The Berezin Transform}
\label{sec:Berezin_transform}
The Berezin transform is an important tool for semiclassical analysis of Toeplitz operators (a good introduction can be found in the book \cite{VasilevskiBook}).  In a certain sense, it moves in the opposite direction of the Toeplitz quantization; the Berezin transform assigns \textit{functions} on the domain (in our case, this is $G/L$) to each operator on the Bergman space, whereas the Toeplitz quantization assigns \textit{operators} on the Bergman space to each function (symbol) on the domain.  We can take this one step further and define an operator that takes functions on $G/L$ to functions on $G/L$, by taking the Berezin transform of the Toeplitz operator with a given symbol. 

We make the following observations before we define the Berezin transform on $\hpiot$.  Let $x\in G$.  Recall that $K^\lambda_x = d_{\lambda} \sigma^\lambda_{\sigma^\lambda(x)v_\lambda, v_\lambda} \in \cA^\lambda$.  Furthermore, under the identification of $\cA^\lambda$ with $\cH_\lambda$, we see that $K^\lambda_x = \sqrt{d_{\lambda}}\sigma^\lambda(x)v_\lambda \in \cH_\lambda$.  Thus, $\langle K^\lambda_x, K^\lambda_x \rangle = ||\sqrt{d_{\lambda}}\sigma^\lambda(x)v_\lambda||^2 = d_{\lambda} ||v_\lambda||^2 = d_{\lambda}$ for all $x\in G$.  
\begin{definition}
   For each operator $S\in\cB(\hplam)$, we define the \textbf{Berezin transform of $S$} to be the function $B^\lambda S:G/L\rightarrow\C$ defined by:
   \[
      xL \mapsto \frac{\langle S K^\lambda_x, K^\lambda_x \rangle}{\langle K^\lambda_x, K^\lambda_x \rangle} =   \langle S \sigma^\lambda(x)v_\lambda, \sigma^\lambda(x) v_\lambda \rangle =  \langle (\siot(x^{-1})S) v_\lambda, v_\lambda \rangle
   \]
   Furthermore, for each $f\in L^1(G/L)$, we define the \textbf{Berezin transform} of $f$ to be the Berezin transform of the operator $T_f$.  That is, $B^\lambda f := B^\lambda (T_f) :G/L\rightarrow \C$.
\end{definition}

The first important property of the Berezin transform is that it is a contraction and a convolution operator:

\begin{proposition}
   \label{prop:Berezin_convolution}
   Let $1\leq p \leq \infty$.  If $f\in L^p(G/L)$, then $B^\lambda f = d_{\lambda} f * |\Delta^\lambda|^2$.  As an operator $B^\lambda : L^p(G/L) \rightarrow L^p(G/L)$, one has that $||B^\lambda||_\op \leq 1$.
\end{proposition}

\begin{proof}
   Note that if $f\in L^\infty(G/L)$, then
   \begin{align*}
      B^\lambda f(x) & = \frac{1}{d_{\lambda}}\langle M_f K^\lambda_x, K^\lambda_x \rangle \\
         & = d_{\lambda}\int_{ G/L } f(y) |\Delta^\lambda(x^{-1}y)|^2 d(y) \\
         & = d_{\lambda}\int_{G/L} f(y) |\Delta^\lambda(y^{-1}x)|^2 d(y) \\
         & = d_{\lambda} f * |\Delta^\lambda|^2.
   \end{align*}
   The density of $L^\infty(G/L)$ in $L^p(G/L)$ (together with Young's inequality) allows us to finish the first claim.  

   Finally, Young's inequality shows that $||f * |\Delta^\lambda|^2||_p \leq ||f||_p ||\, 
   |\Delta^\lambda|^2||_1.$ We note that:
   \[
      ||\, |\Delta^\lambda|^2||_1 = \int_G \sigma^\lambda_{v_\lambda, v_\lambda}(g) \overline{\sigma^\lambda_{v_\lambda, v_\lambda}}(g) dg = \frac{1}{d_{\lambda}} ||v_\lambda||^2 = \frac{1}{d_{\lambda}}
   \]
   by Schur's orthogonality relations.  The fact that $||B^\lambda||_\op \leq 1$ follows immediately.

\end{proof}

As we did for Toeplitz operators, our next step is to write the matrix coefficients for the Berezin transform in terms of the decomposition of $\hpiot$ in irreducible subspaces.  We will also see that $B^\lambda$ is a positive intertwining operator.

Recall the intertwining operator
\begin{align*}
   C: \hpiot &  \rightarrow L^2(G/L)\\
      v\otimes\overline{w} & \mapsto \sigma^\lambda_{v,v_\lambda} \overline{\sigma^\lambda_{w,v_\lambda}}.
\end{align*}
It follows that $C((\hpiot)^\pi) \subseteq L^2(G/L)^\pi$, where $(\hpiot)^\pi$ and $L^2(G/L)^\pi$ denote the respective spaces of $\pi$-isotypic vectors in $\hpiot$ and $L^2(G/L)$.  Furthermore, we know from basic Peter-Weyl theory that:
\[
   L^2(G/L)^\pi = \Span\{\pi_{v,w} \mid v\in \cH_\pi, w\in \cH_\pi^L\} \cong \cH^\pi \otimes \overline{\cH_\pi^L}
\]

\begin{proposition}
   \label{prop:Berezin_matrix_coefficient_formula}
   As above, let $\pi\preceq\siot$, $u\in\cH_\pi$, and $v_0\in\cK^\lambda_\pi$. Then 
   \[
      B^\lambda (\pi_{u,v_0}) = \frac{d_{\lambda}}{d_{\pi}} \pi_{u, \cB_\pi^\lambda v_0},
   \]
   where $\cB^\lambda_\pi: \cK^\lambda_\pi \rightarrow \cK^\lambda_\pi$ is given by
\[
   \cB_\pi^\lambda =  \sum_{k=1}^{m(\pi,\siot)} (\viot)_k^\pi \otimes \overline{(\viot)_k^\pi}, 
\]
Furthermore, $\cB^\lambda_\pi$ is invertible for all $\pi\preceq\siot$. From its definition, it is clear that $\cB^\lambda$ is a positive operator.  It follows that the Berezin transform $B^\lambda :\cK^\lambda \rightarrow \cK^\lambda$ is an invertible, positive $G$-intertwining operator.  Furthermore, $\ker B^\lambda = (\cK^\lambda)^\perp$.  
\end{proposition}
\begin{proof}
   We note that:
   \begin{align*}
      B^\lambda (\pi_{u,v_0})(x) & =  \frac{d_{\lambda}}{d_{\pi}} \langle T_{\pi_{u,v_0}} \sigma^\lambda(x) v_\lambda , \sigma^\lambda(x) v_\lambda \rangle\\
                                  & =  \frac{d_{\lambda}}{d_{\pi}} \sum_{k=1}^{m(\pi,\siot)} \langle u , (\siox{x})^\pi_k \rangle \langle (\viot)^\pi_k, v_0  \rangle \\
                                  & =  \frac{d_{\lambda}}{d_{\pi}} \sum_{k=1}^{m(\pi,\siot)} \langle u, ((\siot)(x)\viot)^\pi_k \rangle \langle  (\viot)^\pi_k, v_0 \rangle \\
                                  & =  \frac{d_{\lambda}}{d_{\pi}} \sum_{k=1}^{m(\pi,\siot)} \langle u, \langle v_0, (\viot)^\pi_k \rangle \pi(x) (\viot)^\pi_k\rangle \\
                                  & = \frac{d_{\lambda}}{d_{\pi}} \left\langle u, \pi(x) \sum_{k=1}^{m(\pi,\siot)}   \langle v_0, (\viot)^\pi_k \rangle (\viot)^\pi_k \right\rangle.
   \end{align*}

We note that the operator $\cB^\lambda_\pi$ is invertible, because $(\viot)^\pi_1,\ldots, (\viot)^\pi_{d(\pi,\siot)}$ is a basis for $\cK^\lambda_\pi$.  It follows that $B^\lambda: \cK^\lambda \rightarrow \cK^\lambda$ is an invertible $G$-intertwining operator. 
\end{proof}
We immediately have the following identification of the image of $L^2(G/L)^\pi$ under the Berezin transform:
\begin{corollary}
   \label{cor:Berezin_transform_pi_isotypic_vectors}
   The Berezin transform sends $L^2(G/L)^\pi$ to $L^2(G/L)^\pi$, and
   \[
      B^\lambda(L^2(G/L)^\pi) = \Span\{\pi_{v,w}\mid v\in \cH_\pi, w\in \cK^\lambda_\pi \} \cong \cH_\pi \otimes \overline{\cK^\lambda_\pi}
    \]
\end{corollary}

In the introduction to this section, we remarked that the Berezin transform for operators is, in a certain sense, ``opposite'' to the Toeplitz quantization operator.  This statement is made precise by the next proposition: in fact, the Berezin transform for functions (which takes the Berezin transform of the Toeplitz operator with a given symbol) is a multiple of $(T^\lambda)^* T^\lambda$.

\begin{proposition}
   \label{prop:Toeplitz_product_trace_Berezin_transform}
   Let $f,g \in L^2(G/L)$.  Then 
   \[
      \Tr(T_f T_g^*) = \Tr(T_f T_{\overline{g}}) = d_{\lambda} \langle B^\lambda f, g\rangle.
   \]
   It follows that, in terms of the linear operator $T^\lambda: L^2(G/L) \rightarrow \Hom(\cA^\lambda)$, the Berezin transform can be written as:
   \[
      B^\lambda = \frac{1}{d_{\lambda}} (T^\lambda)^* T^\lambda,
   \]
   where $\Hom(\cA^\lambda)$ is given the standard Hilbert-Schmidt inner product. 

   Furthermore, 
   \[
      \Tr B^\lambda  = d_{\lambda} \text{ and thus } \Tr( (T^\lambda)^* T^\lambda) = (d_{\lambda})^2.
   \]
\end{proposition}

\begin{proof}
   We see that for all $\phi\in\cA^\lambda$ and $x\in G$, 
   \begin{align*}
      (T_fT_{\overline{g}}) \phi(x) & = (d_{\lambda})^2 \int_G f(y) \int_G \overline{g}(z) \phi(z) \overline{\langle \sigma^\lambda(y)v_\lambda, \sigma^\lambda(z) v_\lambda \rangle} dz \, \overline{\langle \sigma^\lambda(x)v_\lambda, \sigma^\lambda(y)v_\lambda \rangle} dy \\
         & = (d_{\lambda})^2\int_G \phi(z) \left(\overline{g}(z)\int_G f(y)  \langle \sigma^\lambda(z)v_\lambda, \sigma^\lambda(y) v_\lambda \rangle \langle \sigma^\lambda(y) v_\lambda, \sigma^\lambda(x)v_\lambda \rangle dy \right) dz.
   \end{align*}
   If $g$ is continuous, it follows that the expression in parentheses (as a function of $x$ and $z$) gives a continuous integral kernel for the operator $T_f T_{\overline{g}}$, and thus its trace is given by the integral along its diagonal:
   \begin{align*}
      \Tr(T_fT_{\overline{g}}) & = (d_{\lambda})^2 \int_G \overline{g}(z) \int_G f(y) |\langle v_\lambda, \sigma^\lambda(z^{-1}y) v_\lambda \rangle |^2 dy \, dz \\
                               & = (d_{\lambda})^2\langle f * |\Delta^\lambda|^2,g \rangle = d_{\lambda} \langle B^\lambda f,g \rangle.
   \end{align*}
   The extension to arbitrary $g\in L^2(G/L)$ follows from the density of continuous functions in $L^2(G/L)$.
   The first part of the result then follows from the fact that $B^\lambda$ is self-adjoint.  

   Next if $f,g\in L^2(G/L)$, we see that
   \[
      \langle T^\lambda f, T^\lambda g \rangle = \Tr(T_f T_{\overline{g}}) = d_{\lambda} \langle B^\lambda f,g \rangle,
   \]
   which proves the second part of the result.

   Finally, we see that $B^\lambda$ is a finite-rank operator and that:
   \[
      B^\lambda f(x) = d_{\lambda} \int_G f(y) |\Delta^\lambda(y^{-1}x)|^2 \, dy
   \]
   for almost all $x\in G$.  Once again using that $B^\lambda$ is a finite-rank (and thus trace-class) operator with continuous integral kernel, we see that its trace is the diagonal integral of the kernel:
   \[
      \Tr B^\lambda = d_{\lambda} \int_G |\Delta^\lambda(x^{-1}x)|^2 dx = d_{\lambda}\int_G 1 \,dx = d_{\lambda}. \qedhere
   \]
\end{proof}

For each $\lambda=\lambda_1 \omega_1 + \cdots +\lambda_r \omega_r + \chi \in\Lambda^+$ (where $\chi\in \widehat{Z(G)}$), we define the \textbf{support} of $\lambda$ to be the set:
\[
   \supp \lambda := \{\alpha_i\in\Pi \mid \lambda_i \neq 0 \} \subseteq \Pi.
\]
Furthermore, for any subset $S\subseteq\pi$, we write:
\[
   \min_S \lambda := \min\{\lambda_i \mid \alpha_i\in S\}, \max_S \lambda := \max\{\lambda_i \mid \alpha_i \in S\}.
\]

Since $B^\lambda$ is given by convolution with the $d_\lambda |\Delta^\lambda|^2$, we refer to $d_\lambda |\Delta^\lambda|^2$ as the \textbf{Berezin kernel} for $\sigma^\lambda$. The following Lemma shows that the Berezin kernel is an idempotent for the convolution algebra $L^1(G/L)$. 

\begin{lemma}
   \label{lem:idempotent_convolution}
   For each $\lambda\in \Lambda^+$, one has that
   \[
      \Delta^\lambda * \Delta^\lambda = \frac{1}{d_\lambda} \Delta^\lambda
   \]
\end{lemma}
\begin{proof}
   We note that for all $x\in G$,
   \begin{align*}
      \int_G \Delta^{\lambda}(y)\Delta^{\lambda}(y^{-1}x) dy 
           & = \int_{G} \langle v_\lambda, \sigma^\lambda(y) v_\lambda \rangle \langle v_\lambda, \sigma^\lambda(y^{-1}x) v_\lambda \rangle dy \\
           & = \int_G \langle v_\lambda, \sigma^\lambda(y)v_\lambda \rangle \overline{\langle \sigma^\lambda(x)v_\lambda, v_\lambda} \rangle dy \\
           & = \frac{1}{d_\lambda} \langle v_\lambda, \sigma^\lambda(x) v_\lambda \rangle = \frac{1}{d_\lambda} \Delta^\lambda(x)
   \end{align*}
   by the Schur Orthogonality Relations.
\end{proof}

\section{The Berezin Transform and Approximate Identities}
\label{sec:approximate_identity}
In this section, we prove that when a sequence $\{\lambda_n\}_{n\in\N} \in\Lambda^+$ of highest weights satisfies certain growth conditions, the Berezin transform $B^{\lambda_n}$ converges to the identity on $L^2(G/T)$ in the strong operator topology.  Furthermore, we prove some consequences of this result, which include a Szegő limit-type theorem.
\begin{theorem}
   \label{thm:approx_identity}
   Fix a subset $S\subseteq \Pi$.  Then, in the notation and terminology of Section~\ref{sec:parabolic_subgroups}, we consider the corresponding parabolic subgroup $P$ and the group $L = P\cap G$.   Recall that a weight $\lambda\in i\gt^*$ belongs to $i(\gt\ominus\gt_{S})^*$ if and only if $\supp \lambda \subseteq \Pi\backslash S$.

   Now let $\{\lambda_n\}_{n\in\N} \subseteq \Lambda^+\cap i(\gt\ominus\gt_{S})^*$ be a sequence such that  such that for all $C>0$, 
   \[
      \min_{\Pi\backslash S} \lambda_n \geq C\log n
   \]
   and such that there exists $D>0$ such that
   \[
      \max_{\Pi\backslash S} \lambda_n \leq n^D
   \]
   for all sufficiently large $n\in\N$ (that is, the minimum coefficient of $\lambda_n$ over its support grows superlogarithmically and its maximum coefficient grows at most polynomially). 

   \begin{enumerate}
      \item The functions
      \begin{align*}
         |\Delta^{\lambda_n}|^2:=G/L & \rightarrow \C \\
         y & \mapsto d_{\lambda_n} |\Delta^{\lambda_n}(y)|^2 
      \end{align*}
      form an approximate identity as $n\rightarrow \infty$. 
   
      \item If, in addition, the representation $\sigma^{\lambda_n}$ is self-conjugate (that is, equivalent to its contragredient) for all $n\in\N$, then, for each fixed $k\in\N$, the functions
      \begin{align*}
         h_k^{\lambda_n}:G/L \times \cdots \times G/L & \rightarrow \C \\
         (y_1, \ldots, y_k) & \mapsto (d_{\lambda_n})^k\Delta^{\lambda_n}(y_1) \Delta^{\lambda_n}(y_1^{-1}y_2) \cdots \Delta^{\lambda_n}(y_{k-1}^{-1}y_k) \Delta^{\lambda_n}(y_k^{-1})
      \end{align*}
      form an approximate identity as $n\rightarrow \infty$. (Note that $h_1 = |\Delta^{\lambda_n}|^2$.)

   \end{enumerate}
\end{theorem}

\begin{observation}
   It may be possible prove a version of (2) without the assumption that $\sigma^{\lambda_n}$ is self-conjugate for all $n\in\N$.  One needs to prove that the $L^1$ norms of $h^{\lambda_n}_k$ are uniformly bounded.  The main obstruction is that the obvious Cauchy-Schwartz type estimate for the $L^1$ norm of $h^{\lambda_n}_k$ gives 
   \[
      ||h^{\lambda_n}_k||_1 \leq (d_{\lambda_n})^{(k-1)/2},
   \]
   which still grows to infinity as $n\rightarrow\infty$ if $k>1$.
\end{observation}

\begin{proof}
   We first show that $|\Delta^{\lambda_n}|^2$ forms an approximate identity as $n\rightarrow\infty$.  Let $\lambda\in \Lambda^+\cap i (\gt\ominus\gt_S)^*$.  We start by noting that $|\Delta^\lambda|^2$ has unit integral over $G/L$ (as usual, we identity functions on $G/L$ with right $L$-invariant functions on $G$):
   \begin{align*}
      \int_G d_\lambda \Delta^{\lambda}(y) \Delta^{\lambda}(y^{-1}) dy 
           & = d_\lambda \int_G \langle v_{\lambda}, \sigma^\lambda(y) v_{\lambda} \rangle \overline{ \langle v_{\lambda}, \sigma^\lambda(y) v_\lambda \rangle } dy \\
           & = ||v_\lambda||^2 ||v_\lambda||^2 = 1 
   \end{align*}
   by the Schur orthogonality relations. 

   We write $\widetilde{S}^c = \{ 1\leq i \leq n \mid \alpha_i\in \Pi \backslash S\}$.  Now fix $\mu = \sum_{i\in\widetilde{S}^c } \omega_i$. By Lemma~\ref{lem:conical_tensor_prod}, we see that 
   $\Delta^\mu = \prod_{i\in\widetilde{S}^c} \Delta_i$, where we remind the reader that $\Delta_i := \Delta^{\omega_i}$.  Finally, by Corollary~\ref{cor:absolute_value_equals_one}, combined with the continuity of $\Delta^\mu$ and the compactness of $G/L$, one has that for each open neighborhood $V$ of $eL$ in $G/L$, we see that there is $\delta <1$ such that 
   \[
      \sup_{yL \in (G/L) \backslash V} |\Delta^\mu(y)| \leq \delta 
   \]
   Note that if $\lambda\in \Lambda^+\cap i (\gt\ominus\gt_S)^*$ and $\lambda = \lambda_1\omega_1 + \cdots \lambda_r \omega_r+\omega_Z$, then 
   \[
      \Delta^\lambda(y) = \omega_Z(y)\prod_{i=1}^r (\Delta_i)^{\lambda_i}(y) = \prod_{i\in \widetilde{S}^c} (\Delta_i)^{\lambda_i}(y)
   \]
   and thus
   \[
      |\Delta^\lambda(y)| \leq |\Delta^\mu(y)|^{m_\lambda} \leq \delta^m,
   \]
   where $m_\lambda = \min_{\Pi \backslash S}\lambda$. Next, the famous Weyl dimension formula tells us that
   \[
      d_\lambda = \frac{\prod_{\alpha\in\Delta^+} \langle \lambda + \rho, \alpha \rangle}
      {\prod_{\alpha\in\Delta^+}\langle \rho, \alpha \rangle} 
   \]
   is a polynomial of degree $\#\Delta^+$ (the number of positive roots in $\Delta^+(\gg_\C, \gt_\C)$) in the entries of $\lambda$.  In particular, if $M_\lambda=\max_{\Pi\backslash S} \lambda_n$, then:
   \[
      d_\lambda \leq M_\lambda^{\# \Delta^+} 
   \]

   We now consider a sequence $\{\lambda_n\}_{n\in\N} \subseteq \Lambda^+\cap i(\gt\ominus\gt_{S})^*$ as in the hypothesis to the lemma, in which for all $C>0$ there exists $N\in \N$ such that
   \[
      \min_{\Pi\backslash S} \lambda_n \geq C\log n
   \]
   for all $n\geq N$.  We then have, for all $xL \in (G/L) \backslash V$, and for all $n\in \N$ sufficiently large, that:
   \begin{align*}
      d_{\lambda_n} |\Delta^\lambda(x)|^2 & \leq d_{\lambda_n} \delta^{2C\log n} \\
                      & = d_{\lambda_n} n^{2C \log \delta}  \\
                      & \leq n^{(\#\Delta^+) D} n^{2C \log \delta},
   \end{align*}
   where we have chosen $D$ such that $\max_{\Pi\backslash S} \lambda_n \leq n^D$ for sufficiently large $n\in \N$.  Since $0 < \delta <1$, we see that $\log \delta<0$.  It is sufficient to choose $C>0$ large enough such that $(\#\Delta^+)D - 2C \log \delta <0$, in which case one sees that $d_{\lambda_n} n^{2C\log \delta} \rightarrow 0$ as $n\rightarrow\infty$. 

   We now prove part (2) of the theorem.  Let $\lambda\in \Lambda^+\cap i (\gt\ominus\gt_S)^*$.  We first show that $h^k_\lambda$ is indeed right-$L$ invariant in each of its $k$ input variables.  Recall that 
   \[
      \Delta^\lambda(l_1 x l_2) = \langle \sigma^\lambda(l_1^{-1})v_\lambda, \sigma^\lambda(x) \sigma^\lambda(l_2)v_\lambda\rangle = \chi_\lambda(l_1^{-1}l_2^{-1}) 
   \]
   for all $x\in G$ and $l_1,l_2 \in L$.  We note that for all $y_1,\ldots,y_k\in G$ and $l_1,\ldots,l_n\in L$, 
   \begin{align*}
      h_k^\lambda(y_1 l_1, \ldots y_k l_k) 
                   & = (d_{\lambda})^k \Delta^{\lambda}(y_1l_1)\Delta^{\lambda}(l_1^{-1}y_1y_2l_2) \cdots \Delta^{\lambda}(l_{k-1}^{-1}y_{k-1}^{-1}y_kl_k) \Delta(l_k^{-1}y_k^{-1}) \\
                   & = d_{\lambda}^k \chi_\lambda(l_1l_1^{-1} \cdots l_kl_k^{-1}) \Delta^{\lambda}(y_1)\Delta^\lambda(y_1^{-1}y_2)\cdots \Delta^\lambda(y_{k-1}^{-1}y_k) \Delta^\lambda (y_k^{-1}) \\
                   & = h_k^\lambda(y_1,\ldots, y_k).
   \end{align*}

   Next, we show that $h_k^\lambda$ has unit integral.  We merely need note that:
   \begin{align*}
      \int_G \cdots & \int_G h_k^\lambda(y_1,\ldots,y_k) dy_1 \ldots dy_k \\
                    &= (d_\lambda)^k\int_G \cdots \int_G \Delta^{\lambda}(y_1) \Delta^{\lambda}(y_1^{-1}y_2) \cdots \Delta^{\lambda}(y_{k-1}^{-1}y_k) \Delta^{\lambda}(y_k^{-1}) dy_1 \cdots dy_k \\
                    & = (d_\lambda)^k (\Delta^\lambda * \cdots * \Delta^\lambda)(e)\\
                    & = 1
   \end{align*}
   by Lemma~\ref{lem:idempotent_convolution}, where we have taken a convolution of $k+1$ copies of the function $\Delta^\lambda$.

   The next step is to prove that for each neighborhood $V$ of $(eL,\ldots,eL)\in G/L\times \cdots \times G/L$, the functions $h_k^{\lambda_n}$ uniformly converge to zero on its complement $V^c$ in $G/L\times \cdots \times G/L$.  As usual, we will identify $V$ with the corresponding right-$L$-invariant neighborhood of $(e,\ldots, e) \in G\times \cdots\times G$.  We begin by noting that $|h_k^\lambda(y_1, \ldots, y_k)|\leq (d_\lambda)^k$ for all $y_1, \ldots, y_k\in G$.  Furthermore, if $|h_k^\lambda(y_1, \ldots y_k)|=(d_\lambda)^k$, then 
   \[
      |\Delta^\lambda(y_1)| = |\Delta^\lambda(y_1^{-1}y_2)|=  \cdots = |\Delta^\lambda(y_{k-1}^{-1})|=1,
   \]
   and $|\Delta^\lambda(y_k^{-1})|=1$.  
   It follows from Corollary~\ref{cor:absolute_value_equals_one} that $y_1, y_1^{-1}y_2, \ldots, y_{k-1}^{-1}y_k\in L$. Since $L$ is a group, we finally have that $y_1, \ldots, y_k\in L$.

   As above, let $\mu = \sum_{i\in\widetilde{S}^c} \omega_i$. We note, as before, by the continuity of $h_k^\mu$, that there is $\delta < 1$ such that 
   \[
      \sup_{(y_1, \ldots, y_k)\in (G\times \cdots \times G) \backslash V} |h_k^\mu(y_1, \ldots, y_k)| \leq \delta (d_\lambda)^k.
   \]
   Furthermore, we note that:
   \[
      \frac{1}{(d_\lambda)^k}|h_k^\lambda(y_1, \ldots, y_k)| = \prod_{i\in\widetilde{S}^c}
      \frac{1}{(d_{\omega_i})^k} |h_k^{\omega_i}(y_1, \ldots, y_k)|^{\lambda_i},
   \]
   where $\lambda = \lambda_1 \omega_1 + \cdots + \lambda_k \omega_k + \chi_Z$. Also, note that $c:=\frac{1}{\prod_{i\in\widetilde{S}^c} (d_{\omega_i})^k}$ is a constant that does not depend on $\lambda$.  Thus, if we let $m_\lambda=\min_{\Pi\backslash S} \lambda$, then:
   \begin{align*}
      |h_k^\lambda(y_1, \ldots, y_k)| & \leq c(d_\lambda)^k|h_k^\mu(y_1,\ldots, y_k)|^{m_\lambda} \\
            & \leq c(d_\lambda)^k \delta^{m_\lambda}.
   \end{align*}
   for all $(y_1, \ldots, y_k) \in (G\times \cdots \times G) \backslash V$. Since this expression grows polynomially in $\lambda$ and converges to zero exponentially in $m_\lambda$, one shows, using essentially the same argument as above, that if $\{\lambda_n\}_{n\in\N} \subseteq \Lambda^+ \cap i(\gt \ominus \gt_S)^*$ satisfies the hypotheses of the theorem, then 
   \[
      \sup_{(y_1,\ldots, y_k) \in (G\times \cdots \times G) \backslash V} |h_k^{\lambda_n}(y_1, \ldots, y_k)| \stackrel{n\rightarrow\infty}{\rightarrow} 0.
   \]
   %there must be a neighborhood $W$ of $eL$ such that $W\times \cdots \times W \subseteq V$.  Finally, there must be a right-$L$-invariant neighborhood $U$ of $e$ in $G$ such that $xyL\in W$ whenever $x,y\in U$. 

   %Now suppose that $x_1L, \cdots, x_kL\in G/L$ and that for some $1\leq i\leq k$, one has that $x_i\notin U$.  Let $j$ be the first index such that $x_i\notin U$.  If $j=1$, then $x_1\notin U$.  If $1<j$, then we see that $x_{j-1}^{-1}x_j\notin U$, because otherwise one would have that $x_jL = x_{j-1}(x_{j-1}^{-1}x_j)L\in W$.

   %It follows that if $(x_1L, \ldots, x_kL)\in G/L \times \cdots \times G/L$ and $(x_1L, \ldots, x_kL) \notin V$, then it is not true that $x_1\in U$, $x_1^{-1}x_2\in U$, \ldots, $x_{k-1}^{-1}x_k\in U$, and $x_k^{-1}\in U$. By the proof of part (1), we know that  
   %    \[
   %       \sup_{x\in G\backslash U} |\Delta^{\lambda_n}(x)| \stackrel{n\rightarrow \infty}{\rightarrow} 0.
   %    \]
   %Supposing that $j>1$ is such that $x_{j-1}^{-1}x_j\notin U$, we see that 
   %\[
   %   |h_k^{\lambda_n}(x)| \leq |\Delta^\lambda(x_j^{-1}x_j)| \leq \sup_{x\in G\backslash U} |\Delta^{\lambda_n}(x)|
   %\]
   %for all $n\in\N$, and thus 

   The astute reader will note that we have not yet used the hypothesis that $\sigma^{\lambda_n}$ is a self-conjugate representation.  We will now use this hypothesis for the last part of the proof that $h_k^{\lambda_n}$ is an approximate identity as $n\rightarrow\infty$.  Since the $h_k^{\lambda_n}$'s are not nonnegative functions, we must additionally show that their $L^1$ norms are uniformly bounded.  The assumption that $\sigma^{\lambda_n}$ is self-conjugate shows that $\langle v_\lambda, \sigma^{\lambda_n}(x) v_\lambda \rangle = \overline{\langle v_\lambda, \sigma^{\lambda_n}(x)v_\lambda\rangle}$ for all $x\in G$; that is, the functions $\Delta^{\lambda_n}$ (and thus $h_k^{\lambda_n}$) are real valued for all $n\in\N$.

   By the continuity of the functions $h^{\lambda_n}_k$, one sees that there is a right-$L$-invariant neighborhood $W$ of $(e, \cdots, e)$ in $G\times \cdots \times G$ such that:
   \[
      h_k(y_1,\ldots, y_k) > 0.
   \]
   for all $(y_1, \ldots y_k)\in W$.  

   Now let $(h^{\lambda_n}_k)^- = \min{0,h^{\lambda_n}_k}$ for each $n\in \N$.  We see that $|h^{\lambda_n}_k| = h^{\lambda_n}_k+ 2(h^{\lambda_n}_k)^-$.  Since $(h^{\lambda_n}_k)^-=0$ on the open set $W$, we see that 
   \[
      \lim_{n\rightarrow \infty} ||(h^{\lambda_n}_k)^-||_{\infty} = 0.
   \]
   It follows that 
   \[
      \lim_{n\rightarrow \infty} ||h^{\lambda_n}_k||_1 = 1,
   \]
   and thus that the $L^1$ norms of $h^{\lambda_n}_k$ are uniformly bounded as $n\rightarrow \infty$. 
\end{proof} 

We are now ready to show that the Berezin transform $B^{\lambda_n}$ converges to the identity operator on $L^2(G/L)$ as $n\rightarrow\infty$ under the hypotheses of the previous theorem. 
\begin{corollary}
   \label{cor:asymptotic_Berezin}
   Fix a subset $S\subseteq \Pi$.  Then, in the notation and terminology of Section~\ref{sec:parabolic_subgroups}, we consider the corresponding parabolic subgroup $P$ and the group $L = P\cap G$.   Recall that a weight $\lambda\in i\gt^*$ belongs to $i(\gt\ominus\gt_{S})^*$ if and only if $\supp \lambda \subseteq \Pi\backslash S$.

   Now let $\{\lambda_n\}_{n\in\N} \subseteq \Lambda^+\cap i(\gt\ominus\gt_{S})^*$  be a sequence satisfying the hypotheses of Theorem~\ref{thm:approx_identity}.
   Then:
   \begin{enumerate}
      \item For all $1\leq p < \infty$ and all $f\in L^p(G/L)$,
      \[
         B^{\lambda_n} f \stackrel{n\rightarrow\infty}{\longrightarrow} f \text{ (in $L^p(G/L)$)},
      \]
      The same conclusion also holds (with uniform convergence) if $f\in C(G/L)$.

   \item Suppose furthermore that $\sigma^{\lambda_n}$ is self-conjugate for all $n\in\N$. Let $f_1, \ldots, f_k\in L^p(G/L)$ If $1\leq p < \infty$ and $f,g\in L^p(G/L)$, then:
         \[
            B^{\lambda_n}(T^{\lambda_n}_{f_1} \cdots T^{\lambda_n}_{f_k}) \stackrel{n\rightarrow\infty}{\longrightarrow} f_1 \cdots f_k \text{ (in $L^p(G/L)$)} 
         \]
         The same conclusion also holds (with uniform convergence) if $f_1, \ldots, f_k\in C(G/L)$.
   \end{enumerate}
\end{corollary}
\begin{proof}
   Part (1) is an immediate corollary of Theorem~\ref{thm:approx_identity} and the fact that $B^\lambda f = f * (d_\lambda |\Delta^\lambda|^2)$.

   Now suppose that $f_1,\ldots, f_k\in L^p(G/L)$.  We define the function $f_1\otimes \cdots f_k \in L^p(G/L \times \cdots \times G/L)$ by $(f_1\otimes \cdots \otimes f_k)(x_1, \ldots, x_k) = f_1(x_1) \cdots f_k(x_k)$ for $x_1, \ldots, x_k\in G$. Then
   \begin{align*}
      \langle T^\lambda_{f_1}\cdots  T^\lambda_{f_k} &  \sigma^\lambda(x) v_\lambda, \sigma^\lambda(x) v_\lambda \rangle \\ 
            = & \, d_\lambda \int_G f_1(y_1) \langle T_{f_2} \cdots T_{f_k} \sigma^\lambda(x) v_\lambda , \sigma^\lambda(y_1)v_\lambda\rangle \overline{\langle \sigma^\lambda(x) v_\lambda, \sigma^\lambda(y_1) v_\lambda \rangle }\, dy_1 \\
         = & \, (d_\lambda)^k\int_G \cdots \int_G f_1(y_1) \cdots f_k(y_k)
              \langle \sigma^\lambda(x) v_\lambda, \sigma^\lambda(y_1) v_\lambda \rangle \langle \sigma^\lambda(y_1)v_\lambda, \sigma^\lambda(y_2) v_\lambda \rangle \\
           & \cdots \langle \sigma^\lambda(y_{k-1})v_\lambda, \sigma^\lambda(y_k) v_\lambda \rangle \langle \sigma^\lambda(y_k) v_\lambda, \sigma^\lambda(x) v_\lambda \rangle \, dy_k \cdots dy_1\\
              % \int_G g(z) \langle \sigma^\lambda(x) v_\lambda, \sigma^\lambda(z) v_\lambda \rangle \overline{\langle \sigma^\lambda(y) v_\lambda, \sigma^\lambda(z) v_\lambda \rangle}\, dz \\
          % & \overline{\langle \sigma^\lambda(x) v_\lambda, \sigma^\lambda(y)v_\lambda \rangle}\, dy \\
          = & \, (d_\lambda)^k\int_G \cdots \int_G f_1(y_1)\cdots f_k(y_k) \Delta^\lambda(x^{-1}y_1) \Delta^\lambda(y_1^{-1}y_2) \\
            & \cdots \Delta^\lambda(y_{k-1}^{-1} y_k) \Delta^\lambda(y_k^{-1}x)\,  dy_k \cdots dy_1 \\
          = & \, (d_\lambda)^k\int_G \cdots \int_G f_1(y_1)\cdots f_k(y_k) \Delta^\lambda(y_1^{-1}x) \Delta^\lambda(y_2^{-1}y_1) \\
            & \cdots \Delta^\lambda(y_k^{-1} y_{k-1}) \Delta^\lambda(x^{-1}y_k)\,  dy_k \cdots dy_1 \\
         = & \, (f_1 \otimes \cdots \otimes f_k) * h_k^\lambda(x, \cdots, x) 
   \end{align*}

   The result then follows from Theorem~\ref{thm:approx_identity}.
\end{proof}

Corollary~\ref{cor:asymptotic_Berezin}, in turn, allows us to easily prove the following result about the asymptotic multiplicity of an irreducible representation $\pi$ in $\siot$ as $\lambda$ tends to infinity as in the hypotheses to Theorem~\ref{thm:approx_identity}. In fact, for $\lambda\in\Lambda^+ \cap i(\gt\ominus\gt_S)^*$ sufficiently large, one has that $\pi$ appears in $\siot$ with multiplicity $\dim \cH_\pi^L$.

\begin{corollary}
   \label{cor:asymptotic_tensor_product}
   Fix $S\subseteq \Pi$ and let $\{\lambda_n\}_{n\in\N} \subseteq \Lambda^+\cap i(\gt\ominus\gt_S)^*$ be a sequence satisfying the hypotheses of Theorem~\ref{thm:approx_identity} such that each representation $\sigma^{\lambda_n}$ is self-conjugate. Now let $\pi$ be a representation such that $\cH_\pi^L\neq\{0\}$.  Then, for $n\in\N$ sufficiently large, 
   \[
      m(\pi,\sniot) = \dim \cH_\pi^L.
   \]
   % Furthermore, if we let 
   % \[
   %    s_n = \max_{1\leq i\neq j \leq m(\pi,\sigma^{\lambda_n})} |\langle (\iot{v_{\lambda_n}}{v_{\lambda_n}})^\pi_i, (\iot{v_{\lambda_n}}{v_{\lambda_n}})^\pi_j \rangle|,
   % \]
   % then 
   % \[
   %    \lim_{n\rightarrow\infty} s_n = 0.
   % \]
   % That is, the basis $\{(\vniot)^\pi_1, \ldots,(\vniot)^\pi_{m(\pi,\sniot)}\}$ for $\cK^{\lambda_n}_\pi$ becomes approximately orthonormal as $n\rightarrow\infty$.
\end{corollary}
\begin{proof}

   By Proposition~\ref{cor:Berezin_transform_pi_isotypic_vectors}, we know that
   \begin{equation}
      \label{eq:Berezin_pi_isotypic_image}
      \cB^\lambda(L^2(G/L)^\pi) = \cH_\pi \otimes \overline{\cK^\lambda_\pi} \subseteq \cH_\pi \otimes \overline{\cH_\pi^L} = L^2(G/L)^\pi.
   \end{equation}

   By Corollary~\ref{cor:asymptotic_Berezin}, we know that
   \[
      \lim_{n\rightarrow \infty} B^{\lambda_n}(\pi_{w,v}) = \pi_{w,v}.
   \]
   It follows that for all $n\in\N$ sufficiently large, one has that $\langle \cB^{\lambda_n}\pi_{w,v}, \pi_{w,v} \rangle\neq 0$ (to see this, recall that if $f\in L^2(G/L)^\pi\backslash\{0\}$ then $||f-g||_2 \geq ||f||_2>0$ for all $g\in L^2(G/L)^\pi$ such that $g\perp f$)). 

   Fix a basis $f_1, \ldots, f_s$ for the finite-dimensional vector space $L^2(G/L)^\pi$ (where $s= \dim L^2(G/L)^\pi = \dim \cH_\pi \dim \cH_\pi^L$). We see that for $n\in\N$ sufficiently large, one has that $\langle \cB^\lambda(f_k), f_k\rangle\neq 0$ for all $1\leq k\leq s$.  It follows that for $n\in\N$ sufficiently large, $\cB^{\lambda_n}(L^2(G/L)^\pi = L^2(G/L)^\pi$.

   By (\ref{eq:Berezin_pi_isotypic_image}), we conclude that $\cK^{\lambda_n}_\pi \cong \cH_\pi^L$.  Since $\dim \cK^\lambda_\pi = m(\pi, \sniot)$, we are done.
\end{proof} 

Corollary~\ref{cor:asymptotic_tensor_product} allows us to consistently decompose the representations $\hniot$ into irreducible subrepresentations for large $n\in\N$ in the following way.  For large $n\in\N$, we have that the intertwining operator
\begin{align*}
   C_n: \hniot & \rightarrow L^2(G/L) \\
   v\otimes\overline{w} & \mapsto \sigma^{\lambda_n}_{v,v_{\lambda_n}} \overline{\sigma^{\lambda_n}_{w,v_{\lambda_n}}}  
\end{align*}
has an isomorphic restriction $C_n|_{(\hniot)^\pi} : (\hniot)^\pi\rightarrow L^2(G/L)^\pi$ to the $\pi$-isotypic vectors.  Thus, we need merely choose an orthonormal basis $\{v_1, \ldots, v_{\dim \cH_\pi^L}\}$ for $\cH_\pi^L$.  We can thus define the irreducible invariant spaces $(\hniot)^\pi_k$ of $\hniot$ by:
\[
   (\hniot)^\pi_k = C^{-1}(\{\pi_{w,v_k}\mid w\in \cH_\pi\}),
\]
and we have that:
\[
   (\hniot)^\pi = \bigoplus_{k=1}^{\dim \cH^L_\pi} (\hniot)^\pi_k.
\]

The advantage gained by this decomposition is that we have made only one arbitrary choice (that of an orthonormal basis for $\cH_\pi^L$), and we have obtained an orthogonal decomposition of $(\hniot)^\pi$ into irreducible copies of $\pi$ for all $n\in\N$ sufficiently large.  
In fact, this decomposition allows us to ask the following question: do the bases $\{(\vniot)^\pi_1,\ldots, (\vniot)^\pi_{\dim \cH^L_\pi}\}$ for $\cK^{\lambda_n}_\pi$ converge to orthonormal bases as $n\rightarrow\infty$?  The fact that the operator
\[
\frac{d_{\lambda_n}}{d_\pi} \cB^{\lambda_n}_\pi = \frac{d_{\lambda_n}}{d_\pi}\sum_{k=1}^{\dim \cH^L_\pi} (\vniot)^\pi_k \otimes \overline{(\vniot)^\pi_k} : L^2(G/L)^\pi \rightarrow L^2(G/L)^\pi
\]
converges to the identity operator as $n\in\N$ (this follows from  Proposition~\ref{prop:Berezin_matrix_coefficient_formula} and Corollary~\ref{cor:asymptotic_Berezin}) suggests that, in fact, the \textit{rescaled} bases $\{\sqrt{\frac{d_{\lambda_n}}{d_\pi}} (\vniot)^\pi_k \otimes \overline{(\vniot)^\pi_k} \}_{k=1}^{\dim\cH_\pi^L}$ become orthonormal in the limit as $n\in\infty$, in the sense that\[
   \lim_{n\rightarrow\infty} \frac{d_{\lambda_n}}{d_\pi}\langle (\vniot)^\pi_j, (\vniot)^\pi_k \rangle = \delta_{jk}
\]
for all $1\leq j,k\leq \dim \cH_\pi^L$.

Finally, we observe that Theorem~\ref{thm:approx_identity} allows us to use the results in \cite{HLW} to prove a Szegő Limit Theorem. A Szegő Limit Theorem for Toeplitz operators $T^\lambda_f$ with a real-valued symbol $f\in L^1(G/L)$ shows that as $\lambda\rightarrow\infty$ in an appropriate manner, the proportion of eigenvalues of $T^\lambda_f$ greater than a given real number $\tau\in\R$ such that $\mu(f^{-1}(\tau))=0$ converges to the measure $\mu(f^{-1}((\tau,\infty)))$ of the set of points in $G/L$ where $f$ takes on values greater than $\tau$.  This remarkable theorem was first proved in the context of Toeplitz operators defined on finite-dimensional subspaces of $L^2(S^1)$ by G.\ Szegő in 1915 (see \cite{Szego}).  
\begin{corollary}
   Fix $S\subseteq \Pi$ and let $\{\lambda_n\}_{n\in\N} \subseteq \Lambda^+\cap i(\gt\ominus\gt_S)^*$ be a sequence satisfying the hypotheses of Theorem~\ref{thm:approx_identity}. We refer to the $G$-invariant probability measure on $G/L$ as $\mu$.  For each $f\in L^1(G/L)$ such that $f=\overline{f}$.  For each $\tau\in\R$, let $N^+(\tau, T^{\lambda_n}_f)$ be the number of eigenvalues of $f$ (with multiplicity) that are greater than $\tau$.  If $\tau\in\R$ is such that $\mu(f^{-1}(\tau)) = 0$, then:
   \[
      \lim_{n\rightarrow\infty} \frac{N^+(\tau, T^{\lambda_n}_f)}{\dim \cH^{\lambda_n}} = \mu(f^{-1}((\tau,\infty)))
   \]
\end{corollary}
\begin{proof}
   The authors in \cite{HLW} prove a quite general Szegő limit theorem for Toeplitz operators defined on finite-dimensional subspaces of $L^2(G/K)$, where $G$ and $K\subseteq G$ are compact groups.  In particular, they consider nets $\Sigma$ (with respect to some directed set $(\cR, \geq)$) of invariant subspaces of the form
   \[
      \Sigma(r) = \sum_{\pi\in\sigma(r)} L^2(G/K)^\pi,
   \]
   where $\sigma(r)$ is a finite subset of $\widehat{G/K}$ for each $r\in\cR$ (in turn, we use the standard notation $\widehat{G/K}$ to refer to the set of equivalence classes of irreducible representations $(\pi,\cH_\pi)$ of $G$ such that $\cH_\pi^K\neq\{0\}$). That is, they consider direct sums of isotypic subspaces of $L^2(G/K)$.  As part of their analysis, they define the function $D_\Gamma : G/K\rightarrow \C$ by:
      \[
         D_r(x) = \sum_{\pi\in\sigma(r)} \Tr(\pi(x)P_\pi^L),
      \]
   where $P_\pi^L: \cH_\pi\rightarrow \cH_\pi^L$ is the orthogonal projection.

   Technically, our Toeplitz operators $T^\lambda_f \in\End(\cH_\lambda)$ for $f\in L^1(G/L)$ do not satisfy the assumptions made in the paper.  However, a careful reading of the proofs allows us to observe the following.  First, we use the sequence of invariant subspaces $\cA^{\lambda_n}$ of $L^2(G)$ for $n\in\N$ in place of the net $\Sigma$. Furthermore, in place of the function $D_r$ for $r\in\cR$, we use the function $d_{\lambda_n}|\Delta^{\lambda_n}|^2$. We have proved here that $d_{\lambda_n}|\Delta^{\lambda_n}|^2$ satisfies the conclusions of Lemma~3.2 (these correspond to Proposition~\ref{prop:Berezin_matrix_coefficient_formula} and Lemma~\ref{lem:idempotent_convolution} in our paper) and Lemma~3.3 (corresponding to our Corollary~\ref{cor:asymptotic_Berezin}) in \cite{HLW}. We have also proved that our Toeplitz operators satisfy the conclusions of Lemma~3.5 of \cite{HLW} (corresponding to Propositions~\ref{prop:Berezin_convolution}, \ref{prop:Berezin_matrix_coefficient_formula}, and \ref{prop:Toeplitz_product_trace_Berezin_transform}). Since the proof of the Szegő limit theorem in \cite{HLW} depends only on their Lemmas~3.2, 3.3, and 3.5, the proof of their Theorem~3.10 applies verbatim to prove the Szegő Limit Theorem in our context.

\end{proof}

\section{Concluding Remarks and Further Research}
In this paper we have begun to outline what we hope to be a general program for Toeplitz operators on Bergman spaces over generalized flag manifolds of compact Lie groups.  In the future, one would like to prove more explicit results on the Berezin transform, ideally including an explicit calculation of its spectrum, as in \cite{Zhang}.  As it stands, we currently have only asymptotic results.  This problem is much more difficult than in the case of the compact Hermitian symmetric spaces, because one no longer has that the decomposition of $\siot$ into irreducible subspaces is multiplicity free for all irreducible representations $\sigma^\lambda$.  

Indeed, formulas for even the multiplicities of the irreducible subrepresentations of $\siot$ are already quite non-trivial and combinatorically complex; the authors are not aware of any general formula for decomposing the vector $\viot$ in $\hpiot$ into its components in the irreducible subspaces of $\hpiot$, which one would expect to aide in the study of the Berezin transform.  Closed formulas for the multiplicities do exist, however: one has the Littlewood-Richardson rule for the case of $\U(n)$ and Steinberg's formula for $G$ a general compact connected Lie group.  And there is nothing which suggests that it would actually be impossible to find the spectrum of the Berezin transform in the general case.

Perhaps more tractable is the following problem: if $H$ is a subgroup of $G$ such that $\sigma^\lambda|_H$ is multiplicity free, then we have shown that the Toeplitz operators with $H$-invariant symbols commute.  One would like to explicitly find the spectrum of such Toeplitz operators, as in \cite{QV1, QV2}, especially in the case of $G=\U(n)$, $H=\U(n-1)$ and $G=\SO(n)$, $H=\SO(n-1)$.

\end{document}